\documentclass[10pt]{amsart} 
\usepackage[utf8]{inputenc} % Required for inputting international characters
\usepackage[T1]{fontenc} % Output font encoding for international characters
\usepackage{amsmath}
\allowdisplaybreaks% allows pagebreak inside align environments; works only with amsmath package
\usepackage{amsxtra}
\usepackage{amscd}
\usepackage{amsthm}
\usepackage{amsfonts}
\usepackage{amssymb}
\usepackage{libertine}
\usepackage{tikz-cd}
\usepackage{eucal}
\usepackage[all]{xy}
\usepackage{graphicx}
\usepackage{relsize}
\usepackage{leftidx}%allows left indexing using usual commands
\usepackage[USenglish]{babel}
\usepackage[autostyle=true]{csquotes}
\usepackage[backend=biber, style=alphabetic, maxalphanames=99, natbib=true, maxnames=99]{biblatex}
\usepackage{mathtools}
\usepackage{mathrsfs}
\usepackage{stmaryrd}
\usepackage[bookmarks=false]{hyperref}
\newtheoremstyle{newplain}
  {8pt}% ABOVESPACE
  {8pt}% BELOWSPACE
  {\itshape}  % BODYFONT
  {}       % INDENT (empty value is the same as 0pt)
  {\rmfamily\scshape} % HEADFONT
  {.}         % HEADPUNCT
  {5pt plus 1pt minus 1pt} % HEADSPACE
  {}          % CUSTOM-HEAD-SPEC
\theoremstyle{newplain}
\newtheorem*{theorem*}{Theorem}
\newtheorem{theorem}{Theorem}[section]
\newtheorem{definition}[theorem]{Definition}
\newtheorem{proposition}[theorem]{Proposition}
\newtheorem*{corollary*}{Corollary}
\newtheorem{corollary}[theorem]{Corollary}
\newtheorem{lemma}[theorem]{Lemma}
\theoremstyle{remark}
\newtheorem{remark}[theorem]{Remark}

\newcommand{\bb}[1]{\mathbb{#1}}
\newcommand{\mc}[1]{\mathcal{#1}}
\newcommand{\mf}[1]{\mathfrak{#1}}
\newcommand{\ms}[1]{\mathsmaller{#1}}

\DeclareMathOperator{\GL}{GL}

\DeclareMathOperator{\Un}{U}
\DeclareMathOperator{\Aut}{\mathrm{Aut}}
\DeclareMathOperator{\Hom}{Hom}

\DeclareMathOperator{\Der}{Der}
\DeclareMathOperator{\ad}{Ad}
\DeclareMathOperator{\aad}{ad}

\DeclareMathOperator{\Weyl}{Weyl}

\DeclareMathOperator{\Ima}{Im}
\DeclareMathOperator{\id}{id}
\DeclareMathOperator{\codim}{codim}
\DeclareMathOperator{\ord}{ord}
\DeclareMathOperator{\rk}{rank}
\DeclareMathOperator{\End}{End}

\DeclareMathOperator{\Stab}{Stab}
\DeclareMathOperator{\C}{C}
\DeclareMathOperator{\N}{N}
\newcommand{\Cher}[1][c]{\mathcal{H}_{1, #1, X, G}}
\DeclareMathOperator{\supp}{supp}

\DeclareMathOperator{\Span}{span}

\newcommand{\OO}{\mathcal{O}}

\DeclareMathOperator{\hh}{HH}
\DeclareMathOperator{\hc}{HC}
\DeclareMathOperator{\h}{H}
\DeclareMathOperator{\pr}{pr}

\DeclareMathOperator{\cone}{cone}

\DeclareMathOperator{\Sym}{Sym}
\DeclareMathOperator{\Sh}{Sh}
\DeclareMathOperator{\Conj}{Conj}
\DeclareMathOperator{\ev}{ev}
\DeclareMathOperator{\Ch}{Ch}
\DeclareMathOperator{\tr}{tr}

\newcommand{\pd}[1]{\frac{\partial}{\partial #1}}

\newcommand{\Chered}[2][\mathfrak{h}]{\widehat{H}_{1, \mathsmaller{\lau{\hbar}}}(#1, #2)}

\newcommand{\colim}{\varinjlim}
\newcommand{\invlim}{\varprojlim}
\renewcommand{\*}{\ast}
\DeclareMathOperator{\res}{res}

\newcommand{\coor}[1]{{#1}^{\textrm{coor}}}

\newcommand{\presup}[2]{\prescript{\ms{#1}}{}{\mkern-1.5mu #2}}
\newcommand{\HC}{\mc{A}_{n-l, l}^H}
\newcommand{\fl}{\coor{\pi}_{\*}\OO_{\textrm{flat}}(\coor{\mc{N}}|_{X_H^i}\times\HC)}
\newcommand{\efl}{\coor{\pi}_{\*}\OO_{\textrm{flat}}(\coor{\mc{N}}\times\HC)}
\newcommand{\sfl}{\coor{\pi}_{\*}
\OO_{\textrm{flat}}^{\infty}(\coor{\mc{N}}\times\HC)}

\newcommand{\lau}[1]{(\hspace{-0,1em}\!(\hspace{-0,05em}#1\hspace{-0,05em})\hspace{-0.1em}\!)}

\DeclareMathOperator{\ind}{ind}

\DeclarePairedDelimiter\floor{\lfloor}{\rfloor}
\makeatletter
\DeclareFontEncoding{LS1}{}{}
\DeclareFontSubstitution{LS1}{stix}{m}{n}
\DeclareMathAlphabet{\mathcal}{LS1}{stixscr}{m}{n}
\makeatother

\begin{document}
\title{Trace densities and algebraic index theorems for sheaves of formal Cherednik algebras}
\author{Alexander Vitanov}
%\address{\textsc{Department of Mathematics, MIT}}
%\email{\textsf{avitanov@protonmail.com}}
\begin{abstract}
We show how a novel construction of the sheaf of Cherednik algebras $\mc{H}_{1, c, X, G}$ on a quotient orbifold $Y:=X/G$ in author's prior work leads to results for $\mc{H}_{1, c, X, G}$ which until recently were viewed as intractable. First, for every orbit type stratum in $X$, we define a trace density map for the Hochschild chain complex of $\mc{H}_{1, c, X, G}$, which generalizes the standard Engeli-Felder's trace density construction for the sheaf of differential operators $\mc{D}_X$. Second, by means of the newly obtained trace density maps, we prove an isomorphism in  the derived category of complexes of $\bb C_{Y}\llbracket\hbar\rrbracket$-modules which computes the hypercohomology of the Hochschild chain complex %$\mc{C}_{\bullet}(\mc{H}_{1, \hbar, X, G})$
of the sheaf of formal Cherednik algebras $\mc{H}_{1, \hbar, X, G}$. We show that this hypercohomology 
%of $\mc{C}_{\bullet}(\mc{H}_{1, \hbar, X, G})$ 
is isomorphic to the Chen-Ruan cohomology of the orbifold $Y$ with values in the ring of formal power series $\bb C\llbracket\hbar\rrbracket$. We infer that the Hochschild chain complex of the sheaf of skew group algebras $\mc{H}_{1, 0, X, G}$ has a well-defined Euler characteristic which is proportional to the topological Euler characteristic of $Y$. Finally, we prove an algebraic index theorem. 
\end{abstract}

\maketitle
\tableofcontents
\section{Introduction}
In \cite{Eti04}, Pavel Etingof introduced a global version of the rational Cherednik algebra. Concretely, he attached to every algebraic and analytic variety $X$ with a faithful action of a finite group $G$ with $k$ conjugacy classes of complex reflections a sheaf of Cherednik algebras $\mc{H}_{1, c, X, G}$. He showed that the formal version of that sheaf, $\mc{H}_{1, \hbar, X, G}$, is a formal deformation of the sheaf of skew group algebras $\mc{H}_{1, 0, X, G}=\mc{D}_X\rtimes G$ in the $G$-equivariant topology on $X$, where $\mc{D}_X$ is the sheaf of differential operators on $X$. 
%, in direction of the $g$-fixed point hypersurfaces in $X$. 
Since the introduction of these sheaves of algebras over a decade and a half ago very little advances have been made into their representation theory. One still open problem of particular interest is the derivation of homological invariants under deformations as well as homological "detectors" in the sense of Section $1$ in \cite{RT12} which distinguish between trivial and nontrivial formal deformations of $\mc{D}_X\rtimes G$. In the current paper, we begin filling that gap %by computing the hypercohomology of the Hochschild chain complex of $\mc{H}_{1, \hbar, X, G}$ and 
by showing that the hypercohomology does not distinguish between trivial and nontrivial formal deformations of $\mc{D}_X\rtimes G$ and finally, by proving an algebraic index theorem which serves as a detector of the type of formal deformation of $\mc{D}_X\rtimes G$. 

\subsection*{Main results}
The main results discussed in the current note are an improved version of a part of the research results in the author's PhD thesis \cite{PhDVit19}. 

In \cite{Vit19, PhDVit19}, it is shown how for an $n$-dimensional smooth analytic variety $X$ with a faithful action of a finite group $G$, the sheaf of Cherednik algebras $\mc{H}_{1, c, X, G}$ on the quotient orbifold $Y:=X/G$ can be constructed with tools from Gelfand-Kazhdan formal geometry. Specifically, the sheaf is realized by means of gluing of sheaves of flat sections of special flat holomorphic bundles on the orbit type strata $X_H^i$ of $X$ for all parabolic subgroups $H$ of $G$ in the $G$-equivariant topology on $X$. In the language of \cite{BK04}, these flat holomorphic bundles arise as localizations of special associative algebras $\HC$ with a Harish-Chandra module structure \cite{Vit19, PhDVit19} with respect to certain formal geometric Harish-Chandra torsors. It is shown in \cite{Vit19, PhDVit19} that every section of $\mc{H}_{1, c, X, G}$ over a $G$-invariant Stein open set $U$ in $Y$ corresponds to a family of flat sections of localizations over the orbit type strata in $X$ which have a non-empty intersection with the preimage of $U$ in $X$. The family is uniquely determined by a set of gluing conditions which its members saturate.  

In that note, we extend the localizations on all strata of codimension $l$, $1\leq l\leq n$, to the fixed point submanifolds containing the strata. We utilize the construction from \cite{Vit19, PhDVit19} to derive for every nontrivial parabolic subgroup $H$ of $G$ a map from the sheaf of (formal) Cherednik algebras to the sheaf of flat sections of the localization on the unique connected fixed point submanifold component $X_i^H$ containing $X_H^i$ in the $G$-equivariant topology.  We  show how these maps induces morphisms (see morphisms \eqref{tdm} and \eqref{formaltdm}) 
\begin{align}
\label{tdmone}
\chi^H_i:\mc{C}_{\bullet}(\mc{H}_{1, c, X, G})&\to p_*j_{i*}^H\Omega_{X_i^H}^{2n-2l-\bullet},\\
\label{tdmtwo}
\chi_{i, \hbar}^H:\mc{C}_{\bullet}(\mc{H}_{1, \hbar, X, G})&\to p_*j_{i*}^H\Omega_{X_i^H}^{2n-2l-\bullet}\llbracket\hbar\rrbracket\quad(\hbar:=(\hbar_1, \dots, \hbar_k)~k~\textrm{formal parameters})
\end{align}
from the Hochschild chain complex of sheaves of (formal) Cherednik algebras to the de Rham complex of smooth (formal) differential forms on $X_i^H$ where $p$ is the projection from $X$ onto $Y$. These maps are referred to as \emph{(formal) trace density morphism} because their images at the cohomology level are de Rham cohomology classes which in turn are integrable over compact submanifolds and thus yield traces on   the algebra of global sections. The trace density morphisms in this note follow closely the trace density constructions in \cite{EF08, RT12} and in fact generalize the standard Engeli-Felder trace density construction in \cite{EF08} for the sheaf of holomorphic differential operators $\mc{D}_X$. Namely, in the special case $G=\{\id_G\}$, the trace density morphisms \eqref{tdmone} and \eqref{tdmtwo} reduce to the standard trace density morphism $(2)$ in \cite{EF08} for $\mc{D}_X$. Second, with  the help of trace density morphisms \eqref{tdmtwo}, we construct an isomorphism in the derived category ${\bf{D}}(\bb C_{Y}\llbracket\hbar\rrbracket)$ of complexes of sheaves of modules over the sheaf of rings $\bb C_{Y}\llbracket\hbar\rrbracket$ (see Theorem \ref{quasiiso}) by means of which the hypercohomology of the Hochschild chain complex of $\mc{H}_{1, \hbar, X, G}$ can be calculated. 
%which can be interpreted as a resolution of the Hochschild chain complex $\mc{C}_{\bullet}(\mc{H}_{1, \hbar, X, G})$ of the sheaf of formal Cherednik algebras $\mc{H}_{1, \hbar, X, G}$.
\hypertarget{thma}{\begin{theorem*}[A]
The map of cochain complexes of sheaves
\begin{equation*}
\bigoplus_{\substack{i\\g\in G}}\chi_{i, \hbar}^g: \mc{C}_{\bullet}(\mc{H}_{1, \hbar, X, G})\rightarrow\big(\bigoplus_{\substack{i\\g\in G}}p_*j_{i*}^g\Omega_{X_{i}^g}^{2n-2l_g^i-\bullet}\llbracket\hbar\rrbracket\big)^G
\end{equation*}
on $Y$, where $l_g^i=\codim (X_i^g)$ and $\Omega_{X_{i}^g}^{2n-2l_g^i-\bullet}\llbracket\hbar\rrbracket$ is the de Rham complex of sheaves of smooth formal differential forms on $X_i^g$, is an isomorphism in ${\bf{D}}(\bb C_{Y}\llbracket\hbar\rrbracket)$. 
\end{theorem*}}
By means of Theorem \hyperlink{thma}{(A)}, we %compute the hypercohomology of the Hochschild chain complex of $\mc{H}_{1, \hbar, X, G}$ and 
show that the hypercohomology of the Hochschild chain complex of $\mc{H}_{1, \hbar, X, G}$ is isomorphic to the Chen-Ruan cohomology of the orbifold $Y$ with values in the ring of formal power series $\bb C\llbracket\hbar\rrbracket$ in $k$ indeterminates  (see Corollary \ref{hypercohchenruan}). 
\hypertarget{corb}{\begin{corollary*}[B]
%\label{corb}
There is an isomorphism of $\bb C\llbracket\hbar\rrbracket$-modules $\bb H^{-\bullet}(Y,  \mc{C}_{\bullet}(\mc{H}_{1, \hbar, X, G}))\rightarrow\h_{CR}^{2n-\bullet}(Y, \bb C)\llbracket\hbar\rrbracket$.
\end{corollary*}
}
In the special case $\hbar=0$, Corollary \hyperlink{corb}{(B)} is a generalization of a famous result for the Hochschild homology of the $G$-invariant Weyl algebra in \cite{AFLS00}. Furthermore, this corollary implies that the Hochschild chain complex $\mc{C}_{\bullet}(p_*\mc{D}_X\rtimes G)$ on $Y$ possesses a well-defined Euler characteristic. We show that it is equal to the topological Euler characteristic of $Y$ scaled by the order of $G$ (see Corollary \ref{eulerchr}).
\begin{corollary*}[C]
The Euler characteristic  
$\chi(Y, \mc{C}_{\bullet}(p_*\mc{D}_X\rtimes G))$ is invariant under formal deformations of $p_*\mc{D}_X\rtimes G$ and continuous deformations of $Y$. In particular, $\chi(Y, \mc{C}_{\bullet}(p_*\mc{D}_X\rtimes G))=|G|\cdot\chi(Y)$.
\end{corollary*}
We remark that in a separate manuscript \cite{Vit20a} we adapt the trace density construction discussed here to the holomorphic and cohomological settings and prove subsequently with its help that the sheaf of twisted formal Cherednik algebras is a universal filtered formal deformation of $p_*\mc{D}_X\rtimes G$ on $Y$ -- until recently an open problem which has resisted a rigorous proof for over a decade.

Finally, mimicking the techniques in \cite{FFS05}, \cite{PPT07} and \cite{RT12}, we prove an algebraic index theorem (see Theorem \ref{algindexthm}) identical to \cite[Theorem 6]{RT12} for the sheaf $\mc{H}_{1, \lau{\hbar}, X, G}$ of $1$-parameter formal Cherednik algebras localized at $\hbar$.  
\hypertarget{thmd}{\begin{theorem*}[D]
For~~$\id\in\Gamma(Y, \mc{H}_{1, \lau{\hbar}, X, G})$ and $H\leq G$, $H\neq\{\id_G\}$, the smooth $(2n-2l)$-form \[\chi_{i, \lau{\hbar}}^H(\id)-\hbar^{n-l}\Big(\hat{A}(R_T)\Ch(\frac{-\Theta}{\hbar})\Ch_{\phi^{\hbar}}(\frac{R_N}{\hbar})\Big)_{n-l}\] 
on the codimension $l$ fixed point submanifold $X_i^H$ is exact.
\end{theorem*}}
When $X$ is a compact manifold, one can integrate the differential form in Theorem \hyperlink{thmd}{(D)} over the fixed point submanifold. The integral yields a value for the trace on the global section algebra of the sheaf of formal Cherednik algebras evaluated at the identity. We find out that in contrast to the Euler characteristic, the traces corresponding to the various nontrivial parabolic subgroups of $G$ distinguish between trivial deformations $\mc{D}_X\rtimes G\lau{\hbar}$ and nontrivial deformations $\mc{H}_{1, \lau{\hbar}, X, G}$ of $p_*\mc{D}_X\rtimes G$. 
\subsection*{Outline of the paper}
The rest of the paper is organized as follows. In Section \ref{sec2}, we discuss complex reflections and the definitions of rational Cherednik algebras and sheaves of Cherednik algebras. In Section \ref{sec3}, we compute the Hochschild and cyclic homology of algebras which are needed in Section \ref{sec4} and Section \ref{sec5}. In Section \ref{sec4}, we introduce the trace density morphisms for the sheaf of (formal) Cherednik algebras associated to parabolic subgroups of $G$. We construct a quasi-isomorphism by means of which we calculate the hypercohomology of the Hochschild chain complex $\mc{C}_{\bullet}(\mc{H}_{1, \hbar, X, G})$ of the sheaf of formal Cherednik algebras $\mc{H}_{1, \hbar, X, G}$. We express the Euler characteristic of the Hochschild chain complex of $p_*\mc{D}_X\rtimes G$ in terms of the Euler characteristic of $Y$. In Section \ref{sec5}, we prove an algebraic index theorem. 
%%%%%%%%%%%%%%%%%%%%%%%%
\section{Preliminaries}
\label{sec2}
%%%%%%%%%%%%%%%%%%%%%%%%%%%%%%%%%%%
\subsection{Irreducible well-generated complex reflection groups}
The material in this section is borrowed from the specialized literature on complex reflections. We follow mostly \cite{kane13} and \cite{LT09}. 

Let $\mathfrak{h}$ be a finite $n$-dimensional complex vector space and let $\mathfrak{h}^*$ be its dual.  A semisimple endomorphism $s$ of $\mathfrak{h}$ is called a \emph{complex reflection} in $\mathfrak{h}$ if $\rk(\id_{\mathfrak{h}}-s)=1$. The fixed point subspace $\mathfrak{h}^s:=\ker(\id_{\mathfrak{h}}-s)$ of the complex reflection $s\in\End(\mathfrak{h})$ is a hyperplane called a \emph{reflecting hyperplane} of $s$. 
Suppose that $G$ is a finite subgroup of $\GL(\mathfrak{h})$ and let $\mathcal{S}$ denote the set of complex reflections in $\mathfrak{h}$ contained in $G$. The group $G$ is called a \emph{complex reflection group} if it is generated by $\mathcal{S}$. Given a complex reflection $s\in\mathcal{S}$ in $\mathfrak{h}$ we denote its unique nontrivial eigenvalue by $\lambda_s^{\vee}$ and by $\alpha_s^{\vee}\in\mathfrak{h}$ an eigenvector of $s$ in $\mathfrak{h}$ corresponding to $\lambda_s^{\vee}$ which we call a \emph{root}. Similarly, we designate by $\lambda_s$ the unique nontrivial eigenvalue of $s$ in $\mathfrak{h}^*$ and by $\alpha_s\in\mathfrak{h}^*$ an eigenvector of $s$ in $\mathfrak{h}^*$ corresponding to $\lambda_s$, which we call \emph{coroot}. Since $G$ is finite, all complex reflections $s\in\mathcal{S}$ have a finite order. Hence, the corresponding eigenvalues $\lambda_s^{\vee}$ and $\lambda_s$ are (not necessarily primitive) roots of unity. 

If a complex reflection group $G\subset\GL(\mathfrak h)$ is such that $\mathfrak h$ is a simple left $\bb CG$-module, we call $G$ an \emph{irreducible complex reflection group}. The following theorem shows that the study of complex reflection groups reduces to the study of irreducible complex reflection groups. We formulate the theorem in a slightly more general manner than in \cite{LT09} which suits our purposes in this note better.          
\begin{theorem}[Theorem 1.27, \cite{LT09}]
\label{thm1.1.1}
Suppose that $G$ is a finite complex reflection group on $\mathfrak{h}$. Then $\mathfrak{h}$ is the direct sum of subspaces $\mathfrak{h}_1, \mathfrak{h}_2\dots, \mathfrak{h}_m$ such that the subgroup $G_i$ of $G$, generated by complex reflections whose roots lie in $\mathfrak{h}_i$, acts irreducibly on  $\mathfrak{h}_i$ for every $i=1, \dots, m$, and $G\cong G_1\times G_2\times\dots\times G_m$. If $\mathfrak{u}$ is not fixed pointwise by every element of $G$, then $\mathfrak{u}=\mathfrak{h}_i$ for some $i$.
\end{theorem}
It follows from this theorem that $\mathfrak{h}=\mathfrak{h}^G\oplus\mathfrak{h}_1\oplus\dots\oplus\mathfrak{h}_k$, where the $\mathfrak{h}_i$ are the nontrivial simple left $\bb CG$-modules. The \emph{support} of a complex reflection group $G\subset\GL(\mathfrak{h})$, denoted $\supp(G)$, is the algebraic complement of the subspace $\mathfrak{h}^G$. A direct consequence of Theorem \ref{thm1.1.1} is the ensuing lemma.
\begin{lemma}
\label{supp(G)}
The support of a complex reflection group $G\subset\GL(\mathfrak{h})$ is spanned by the roots of the complex reflections in $G$.
\end{lemma}
\begin{proof}
Theorem \ref{thm1.1.1} yields $\mathfrak{h}=\mathfrak{h}^G\oplus\supp(G)$. Take a vector $v\in\supp(G)$. There exists at least one element $g\in G$ such that $g=s_1\dots s_r$ for some complex reflections $s_1, \dots, s_r$ and  $v\notin\mathfrak{h}^g$. It follows that $v\in\Ima(1-g)$. Then for some $x\in\mathfrak{h}$, we have that 
 \begin{align*}
 v=(1-s_1\dots s_r)x&=x-s_1\dots s_rx\\
 &=(1-s_r)x+s_rx-s_1\dots s_rx\\
 &=(1-s_r)x+(1-s_{r-1})(s_rx)+\dots+(1-s_1)(s_2\dots s_rx)\\
 &=\mu_r\alpha_{s_r}^{\vee}+\mu_{r-1}\alpha_{s_{r-1}}^{\vee}+\dots+\mu_1\alpha_{s_1}^{\vee}
 \end{align*}
 where $\mu_1, \dots, \mu_r\in\bb C$. %$\alpha_{s_i}^{\vee}$ is the root of $s_i$ in $\mf{h}$ for $i=1, \dots, r$. 
Hence, $v\in\Span_{\mc{S}}\{\alpha_s^{\vee}\}$ and consequently $\supp(G)\subset\Span_{\mc{S}}{\alpha_s^{\vee}}$. Conversely, if $v\in\Span_{\mc{S}}\{\alpha_s^{\vee}\}$, then $v\in\supp{G}$ by the fact that by default no root $\alpha_s^{\vee}$ lies in $\mathfrak{h}^G$. Thus, $\Span_{\mc{S}}{\alpha_s^{\vee}}\subset\supp{G}$.   
\end{proof}
The rank of a reflection group $G$, denoted $\rk(G)$, is the dimension of its support. If $G$ has a generating set $\mc{S}$ of complex reflections whose cardinality is equal to the rank of $G$, then $G$ is called a \emph{well-generated} complex reflection group. Irreducible well-generated complex reflection groups $G$ are of particular interest to our work since they admit so called \emph{Coxeter elements}. Let $\mc{S}^{\*}$ denote the set of hyperplanes in $\mathfrak{h}$ fixed by some complex reflection. Set $N:=|\mc{S}|$ and $N^{\*}=|\mc{S}^*|$. The \emph{Coxeter number} of $G$ is the constant $h:=\frac{N+N^{\*}}{\dim\mathfrak{h}}=d_{\dim\mf{h}}$ where $d_{\dim\mf{h}}$ is the highest degree of $G$. A $\zeta$-regular element in $G$ is a group element $g\in G$ with an eigenvalue $\zeta\in\bb C$ and an  eigenvector $v\in\mathfrak h$ for $\zeta$ which is not contained in a hyperplane in $\mc{S}^{\*}$. A \emph{Coxeter element} is a $\zeta_h$-regular element in $G$ where $\zeta_h$ is a primitive $h$-th root of unity. 

The next two lemmas are easy statements for which we did not find a reference. Therefore, we append a short proof thereof. The first lemma gives a condition under which a complex reflection group admits a Coxeter element.
\begin{lemma}
\label{coxeter}
An irreducible well-generated complex reflection group $G$ has a Coxeter element.
\end{lemma}
\begin{proof}
According to \cite[Corollary $31$-$1A$]{kane13} for a primitive $d$-th root of unity $\zeta$ there exists an element $c\in G$ having $\zeta$ as an eigenvalue if and only if $d$ divides a degree $d_i$ for some $i=1, \dots, \dim\mathfrak{h}$. Let $\zeta_h$ be a primitive $h$-th root of unity. Since by definition, $h=d_{\dim\mathfrak{h}}$, there is an element $c\in G$ with $\zeta_h$ as an eigenvalue. Finally, by \cite[Theorem $1.3$, $i)$ and $iii)$]{reiner17}, the group element $c$ is a Coxeter element. 
\end{proof}
The succeeding lemma demonstrates that a Coxeter element in an irreducible well-generated complex reflection group has no eigenvalues equal to one.
 \begin{lemma}
 \label{nouniteigenvalue}
A Coxeter element $c$ in an irreducible well-generated complex reflection group $G\subset\GL(\mathfrak{h})$ with a Coxeter number $h$ is semisimple and has no eigenvalue equal to $1$.
 \end{lemma}
 \begin{proof}
 Semisimplicity of $c$ follows from the fact that linear endomorphisms of finite order over the complex field are diagonalizable. As $c$ is $\zeta_h$-regular, where $\zeta_h$ is a $h$-th primitive root of unity, by \cite[Theorem 4.2, $v)$]{Spr74}, the eigenvalues of $c$ are $\zeta_h^{1-d_1}, \dots, \zeta_h^{1-d_{\dim\mf{h}}}$ where $d_1\leq\cdots\leq d_{\dim\mf{h}}$ are the degrees of $G$. Since by definition, $d_{\dim\mf{h}}=h$, the Coxeter number $h$ is coprime with $1-d_1, \dots, 1-d_{\dim\mf{h}}$. So, no eigenvalue of $c$ is equal to $1$.  
%The definition of a Coxeter element implies the existence of an eigenvector $x$ of $c$ in $\mathfrak{h}$ whose isotropy group $\Stab(x)$ does not contain any complex reflections. On the other hand, Steinberg's fixed point theorem stipulates that $\Stab(x)$ is a complex reflection subgroup. Hence, $\Stab(x)=\{1\}$. Thus the condition $c^hx=\zeta^hx=x$ implies that $c$ is of order $h$. %It is a standart fact from linear algebra that endomorphisms of finite order are diagonalizable over the comlex field, so $c$ is diagonalizable. 
%Furthermore, every eigenvalue of $c$ is an $h$-th order of unity. Also by definition we know that the Coxeter element $c$ possesses a primitive $h$-th root of unity $\zeta_h$. Thus, every eigenvalue of $c$ is of the form $\zeta_h^k$ for $1\leq k<h$.
\end{proof}
%%%%%%%%%%%%%%%%%%%%%%%%
\subsection{Rational Cherednik algbera}
\label{rational}
%%%%%%%%%%%%%%%%%%%%%%%%
To the data $\mathfrak{h}$, $\mathfrak{h}^*$ and $G$ one attaches the rational Cherednik algebra $H_{t, c}(\mathfrak{h}, G)$ (see, e.g., Section $2.6$ in \cite{Eti04}) in the following way.
\begin{definition}
\label{cherednik}
Let $t$ be a complex parameter and let $c\in\mathbb{C}[\mathcal{S}]^{\ad G}$, where $\ad$ refers to the adjoint action of $G$ on itself, be a class function on $\mc{S}$. The rational Cherednik algebra $H_{t, c}(\mathfrak{h}, G)$ is defined as the quotient of the smash-product algebra $T^{\bullet}(\mathfrak{h}\otimes\mathfrak{h}^*)\rtimes G$ by the ideal generated by
\begin{align*}
&gug^{-1}-\presup{g}{u},\quad gyg^{-1}-\presup{g}{y}, \quad[u, u'], \quad [y, y'], \quad[u, y]-t\left(u, y\right)-\sum_{s\in\mathcal{S}}c(s)\left(u, \alpha_s\right)\left(y, \alpha_s^{\vee}\right)s,
\end{align*} 
for all $u, u'\in\mathfrak{h}$, $y, y'\in\mathfrak{h}^*$ where $\presup{g}{(\cdot)}$ denotes the action of $G$ on $\mf{h}$ and $\mf{h}^*$, respectively.
\end{definition}
We can conveniently think of the class function $c$ on $\mc{S}$ as a set of $k$ complex-valued parameters $c_1, \dots, c_k$ where $k$ is the number of conjugacy classes of elements of $\mc{S}$ in $G$. The algebra $H_{t, c}(\mathfrak{h}, G)$ has a natural increasing filtration $F^{\bullet}$, called geometric filtration, given by 
$\deg(\mathfrak{h}^*)=\deg(G)=0$, $\deg(\mathfrak{h})=1$. We define the \emph{degree-wise completion} $\widehat{H}_{t, c}(\mathfrak{h}, G)$ of $H_{t, c}(\mathfrak{h}, G)$ as the scalar extension of the $\bb C[\mf{h}]$-module $H_{t, c}(\mathfrak{h}, G)$ to the ring of formal functions on the formal neighborhood of zero in $\mathfrak{h}$, that is, $\widehat{H}_{t, c}(\mathfrak{h}, G):=\bb{C}[[\mathfrak{h}]]\otimes_{\bb{C}[\mathfrak{h}]}H_{t, c}(\mathfrak{h}, G)$. The completion $\widehat{H}_{t, c}(\mathfrak{h}, G)$ inherits the geometric filtration $F^{\bullet}$ from $H_{t, c}(\mathfrak{h}, G)$ by the rule 
$F^i\widehat{H}_{t, c}(\mathfrak{h}, G):=\bb{C}[[\mathfrak{h}]]\otimes_{\bb{C}[\mathfrak{h}]}F^iH_{t, c}(\mathfrak{h}, G)$. The following theorem is a special case of \cite[Theorem 6.1]{AFLS00}.  
\begin{theorem}
\label{hhcher}
There is an isomorphism of complex vector spaces
\begin{equation*}
HH_j(\mathcal{D}(\mathfrak{h})\rtimes G, \mathcal{D}(\mathfrak{h})\rtimes G)\cong HH^{2n-j}(\mathcal{D}(\mathfrak{h})\rtimes G, \mathcal{D}(\mathfrak{h})\rtimes G)\cong\bb{C}^{a_j},
\end{equation*} 
where $a_j$ is the number of conjugacy classes of elements in $G$ having eigenvalue $1$ with multiplicity $j$. 
\end{theorem}
%%%%%%%%%%%%%%%%%%%%%%%
\subsection{Sheaf of Cherednik algebras}
\label{global}
%%%%%%%%%%%%%%%%%%%%%%%
Assume from now on that $X$ is a $n$-dimensional connected complex manifold equipped with a faithful action of a finite group $G$ of holomorphic automorphisms of $X$ and denote as before $Y:=X/G$. Let $p:X\to Y$ denote the canonical projection map. We define the sheaf of Cherednik algebras, introduced originally by Etingof in \cite{Eti04}, following \cite{FT17}.
  
Let $X^g\subset X$ denote the fixed point set of $g\in G$. A \emph{nonlinear complex reflection} of $X$ is an element $g$ in $G$ such that $X^g$ has a connected component $X_i^g$ of complex codimension $1$ in $X$. In accord with the terminology in \cite{Eti04}, a codimension $1$ connected component $X_i^g\subset X^g$ is called a \emph{reflection hypersurface}. 
We denote by $\mathrm{S}$ the set of pairs $(g, X_i^g)$ of  complex reflections $g$ and codimension $1$ connected components $X_i^g$ of $X^g$ in $X$. Let $c: \mathrm{S}\longrightarrow\bb{C}$ be a $G$-invariant function. Let $D:=\bigcup_{\codim X_i^g=1}X_i^g$ and let $j: X\backslash D\longrightarrow X$ be the open inclusion map. For each $(g, X_i^g)\in\mathrm{S}$, let $\mathcal{O}_X(X_i^g)$ designate the sheaf of holomorphic functions on $X\backslash X_i^g$ taking poles of at most first order only along $X_i^g$ and let $\xi_{X_i^g}: \mc{T}_X\rightarrow\OO_X(X_i^g)/\OO_X$ be the natural surjective map of $\OO_X$-modules. A \emph{Dunkl operator associated to a holomorphic vector field $\mc{Z}$ on $X$ } is a section $\mathbb{D}_{\mc{Z}}$ of the sheaf $p_*j_*j^*(\mathcal{D}_X\rtimes G)$ over an open Stein subset $U$ of $Y$, which %in a $G$-invariant Stein open set $U'\subseteq p^{-1}(U)\subset X$
locally can be written in the form
\begin{equation*} 
\mathbb{D}_{\mc{Z}}=\mathcal{L}_{\mc{Z}} +\sum_{g\in\mathrm{S}}\frac{2c(g)}{1-\lambda_{g}}f_{X_i^g}(g-1).    
\end{equation*}
Here, $\mathcal{L}_{\mc{Z}}$ is the Lie derivative with respect to $\mc{Z}$, the complex number $\lambda_{g}$ is the nontrivial eigenvalue of $g$ on the conormal bundle to the codimension $1$ component $X_i^g$ and $f_{X_i^g}\in\Gamma(p^{-1}(U), \mathcal{O}_X(X_i^g))$ is a function whose residue agrees with $\mc{Z}$ once both are restricted to the normal bundle of $X_i^g$ in $X$, that is $f_{X_i^g}\in\xi_{X_i^g}(\mc{Z})$. Following Section $1.1$ in \cite{FT17}, we attach to the data $X$ and $G$ the ensuing sheaf of associative algebras on Y. 
\begin{definition}
The sheaf of Cherednik algebras $\Cher$ on the orbifold $Y$ is a subsheaf of the sheaf $p_*j_*j^*(\mathcal{D}_X\rtimes G)$ generated locally on $G$-invariant open Stein sets $U$ in $Y$ by $p_*\mathcal{O}_X|_U$, $\bb{C}G$ and Dunkl operators $\bb {D}_{\mc{Z}}$ associated to holomorphic vector fields $\mc{Z}$ on $X$.  
\end{definition}
The definition of $\mc{H}_{1,c, X, G}$ is independent on the choice of a function $f_{X_i^g}$ in the Dunkl operators, because by adding a holomorphic function from $\Gamma(p^{-1}(U), \OO_X)$ to $f_{X_i^g}$, the new Dunkl operator differs from the old one by a section of $p_*\mathcal{O}_X\rtimes G$ over $U$. The sheaf of Cherednik algebras $\Cher$ possesses a natural increasing and exhaustive filtration $\mc{F}^{\bullet}$ which is defined on the generators by $\deg(p_*\mathcal{O}_X)=\deg(\bb{C}G)=0$ and $\deg(\bb{D}_{\mc{Z}})=1$. It is the analogue of the geometric filtration of the rational Cherednik algebra, discussed in Section \ref{rational}. 

We need a topology on $X$ which, when dealing with sheaves,  allows us to switch between $X$ and $Y$ easily. To that aim, we equip $X$ with the $G$-equivariant topology $\mf{T}_X^G$  which is comprised of the $G$-invariant open sets in the analytic topology of $X$ (\emph{cf}. Section $2.1$ in \cite{BM14}). These sets are preimages of open subsets in $Y$. As $p: X\to Y$ is surjective, we have $pp^{-1}(U)=U$ for every open set $U$ in $Y$. Hence, the $G$-equivariant topology $\mf{T}_X^G$ of $X$ is a non-hausdorff topology consisting of open subsets of the form $p^{-1}(p(V))=\bigcup_{g\in G}gV$, where $V$ is open in $X$. The sheaf $\mc{H}_{1,c, X, G}$ can equivalently be viewed as a sheaf on $Y$ and as a sheaf on $X$ in the $G$-equivariant topology.

Now, we define a special basis for the $G$-equivariant topology $\mf{T}_X^G$ on $X$. To that aim, we recall that a \emph{slice} at a point $x\in X$ is a $\Stab(x)$-invariant neighborhood $W_x$ such that $W_x\cap gW_x=\varnothing$ for all $g\in G\setminus \Stab(x)$. A slice is called a \emph{linear} if there is a $\Stab(x)$-invariant open set $V$ in $\bb C^n$ such that $W_x$ is $\Stab(x)$-equivariantly biholomorphic to $V$. As the group $G$ is finite, it acts on $X$ properly discontinuously. Hence, each point $x$ in $X$ possesses a slice $W_x$. Furthermore, by Cartan's Lemma one can always shrink the slice $W_x$ until the $\Stab(x)$-action is linearized. Thus, every point in $X$ possesses a fundamental system of linear slices. Moreover, for every member $W_x'$ of a fundamental system of linear slices at $x$ in $X$ with $\Stab(x):=H$, one can always find a smaller $H$-invariant linear slice $W_x\subset W_x'$ which is $H$-equivariantly biholomorphic to a polydisc in $\bb C^{n-l}\times\bb C^l$ where $\bb C^{n-l}$ is the fixed point subspace of $\bb C^n$ with respect to $H$. Such a set $W_x$ is Stein and the disjoint union of translates $\ind_H^G(W_x):=\coprod_{g\in G/H}gW_x$ is a subset of $\bigcup_{g\in G}gW_x'$. Hence, the collection $\mf{B}_X^G$ of sets $\ind_H^G(W_x)$, where $W_x$ is either an $H$-invariant linear slice biholomorphic to a product of $H$-invariant polydiscs in $\bb C^{n-l}\times\bb C^l$ or an $H$-invariant open subset of the principal stratum $\mathring{X}$ of regular points in $X$ with $gW_x\cap W_x=\varnothing$ for all $g\in G\setminus H$, forms a basis of G-invariant Stein open sets for $\mf{T}_X^G$.  

Note that for every $x$ on the connected component $X_i^H$ of the fixed point submanifold $X^H$, every $H$-invariant linear slice $W_x$ constitutes a holomorphic slice chart for $X_i^H$.  That is, if $x^1, x^2, \dots, x^{n-l}$ are the holomorphic coordinates on the complex vector subspace $\big(\bb C^n\big)^H=\bb C^{n-l}$ and $y^1, \dots, y^l$ are the holomorphic coordinates on the $l$-dimensional complement of $\bb C^{n-l}$ in $\bb C^n$, then $(x^1, \dots, x^{n-l}, y^1, \dots, y^l)$ define local holomorphic coordiantes of $X$ on $W_x$ such that $(x^1, \dots, x^{n-l})$ are local holomorphic coordinates of $W_x\cap X_i^H$ and $(y^1, \dots y^l)$ are local holomorphic coordinates on $W_x$ in transversal direction to $X_i^H$.
%%%%%%%%%%%%%%%%%%%%%%%%%%%%%%%%%%%%%%%%
\section{Hochschild and cyclic homology of some deformation algebras}
\label{sec3}
%%%%%%%%%%%%%%%%%%%%%%%%%%%%%%%%%%%%%%%%
We calculate the Hochschild and cyclic homology of some important for our applications special algebras. We refrain from repeating the well-known definitions and facts on Hochschild and cyclic homology. Instead, we refer the reader to \cite{weibel94} and \cite{Lod13} for a detailed discussion of their theory.

Throughout this section, let $\mc{D}_{\textrm{alg.}}(\mf{h})$ be the algebra of differential operators on $\mf{h}$ with algebraic coefficients and let $\widehat{\mc{D}}(\mathfrak{h})$ denote the degreewise completion of the algebra $\mc{D}(\mathfrak{h})_{\textrm{alg.}}$ with respect to the $\mf{m}$-adic topology on the ring $\bb{C}[\mathfrak{h}]$ where $\mf{m}$ is the maximal ideal in $\bb{C}[\mathfrak{h}]$. When $\mf{h}=\bb C^n$, we use the shorthand notation $\widehat{\mc{D}}_n:=\widehat{\mc{D}}(\bb C^n)$. We need the following technical result for which we were not able to find a reference. 
\begin{proposition}
\label{prophh}
There is an isomorphism of complex vector spaces
\begin{equation*}
%\label{hh}
\hh_{\bullet}(\widehat{\mc{D}}(\mathfrak{h})\rtimes G)\cong\hh_{\bullet}(\mc{D}_{\textrm{alg.}}(\mathfrak{h})\rtimes G).
\end{equation*}
\end{proposition}
\begin{proof}
For brevity, denote by $\C_{\bullet}:=(\C_{\bullet}(\mc{D}_{\textrm{alg.}}(\mathfrak{h})\rtimes G), d)$ the Hochschild chain complex of $\mc{D}_{\textrm{alg.}}(\mathfrak{h})\rtimes G$ with Hochschild differential $d$ and by $\widehat{\C}_{\bullet}:=(\widehat{\C}_{\bullet}(\widehat{\mc{D}}(\mathfrak{h})\rtimes G), d)$ the corresponding completed chain complex. The Lie group $\Un(1)$-action on $\mathfrak{h}$ naturally extends to an action $\rho$ on the differential-graded algebra $\widehat{\C}_{\bullet}$ by algebra automorphisms which commute with the Hochschild differential. Define a mapping $\hat{P}_n: \widehat{\C}_n\longrightarrow\widehat{\C}_n$ by
\begin{equation*}
a_0\otimes\cdots\otimes a_n\mapsto\int_{\Un(1)}\rho(\lambda)(a_0\otimes\cdots\otimes a_n)d\mu(\lambda),
\end{equation*}
where $d\mu(\lambda)$ is the Haar measure on $\Un(1)$. Since the group is compact, the integral is convergent and consequently $\hat{P}_n$ is a well-defined $\bb{C}$-linear map. Its image $\Ima(\hat{P}_n)$ is the fixed point subspace $\widehat{\C}_{n}^{\Un(1)}$ and the kernel $\ker(\hat{P}_n)$ is the algebraic complement of $\widehat{\C}_{n}^{\Un(1)}$ in $\widehat{\C}_n$. By the $\Un(1)$-equivariance of $d$, we have $\hat{P}_{n-1}\circ d_{n}=d_{n-1}\circ\hat{P}_n$. Hence, the mapping $\hat{P}: \widehat{\C}_{\bullet}\longrightarrow\widehat{\C}_{\bullet}$ is a chain complex endomorphism with $d(\widehat{\C}_{n}^{\Un(1)})\subseteq\widehat{\C}_{n-1}^{\Un(1)}$ and $d(\ker(\hat{P}_n))\subseteq\ker(\hat{P}_{n-1})$. This means that $(\widehat{\C}_{\bullet}^{\Un(1)}, d)$ and $(\ker(\hat{P}), d)$ are subcomplexes of $(\widehat{\C}_{\bullet}, d)$ with $(\ker(\hat{P}_{\bullet}), d)$ having a free $\Un(1)$-action. Moreover, as a chain complex $\widehat{\C}_{\bullet}$ is decomposable in a direct sum of the subcomplexes 
\begin{equation}\label{decomp}\widehat{\C}_{\bullet}=\Ima(\hat{P})\oplus\ker(\hat{P})=\widehat{\C}_{\bullet}^{\Un(1)}\oplus\ker(\hat{P}).\end{equation} 

We proceed by showing that the inclusion morphism $i:\widehat{\C}_{\bullet}^{\Un(1)}\hookrightarrow \widehat{\C}_{\bullet}$ is in fact a quasi-isomorphism. We prove that by showing that the corresponding mapping cone chain complex $\cone(i)$ 
\begin{equation*}
\label{chaincone}
\begin{tikzcd}[row sep=tiny]
\cdots\cdots\arrow[rdd, hook, "i_{n+1}"] \arrow[r, "d_{n+1}"]&\widehat{\C}_n^{\Un(1)}\arrow[rdd, hook, "i_{n}"] \arrow[r, "d_{n}"]&\widehat{\C}_{n-1}^{\Un(1)}\arrow[rdd, hook, "i_{n-1}"]\arrow[r, "d_{n-1}"]&\widehat{\C}_{n-2}^{\Un(1)}\arrow[rdd, hook, "i_{n-2}"]\arrow[r, "d_{n-2}"]&\cdots\cdots\\
&\mathlarger{\bigoplus}&\mathlarger{\bigoplus}&\mathlarger{\bigoplus} \\
\cdots\cdots\arrow[r, "d_{n+1}"]&\widehat{\C}_{n+1}\arrow[r, "d_{n+1}", "\mathlarger{D_{n+1}}"']&\widehat{\C}_{n}\arrow[r, "d_n", "\mathlarger{D_n}"']&\widehat{\C}_{n-1}\arrow[r, "d_{n-1}"]&\cdots\cdots
\end{tikzcd}
\end{equation*}
with $\cone(i)_n=\widehat{\C}_{n-1}^{\Un(1)}\oplus\widehat{\C}_n$ and a differential $D_n=\begin{pmatrix}-d_{n-1}&0\\-i_{n-1}&d_n\end{pmatrix}$ is acyclic. To demonstrate this, it suffices to check that $\ker(D_{n})\subseteq\Ima(D_{n+1})$ for every $n\in\bb{Z}_{\geq{0}}$. It is obvious that an arbitrary element $(v, w)\in\ker(D_n)$ satisfies $d_{n-1}(v)=0$ and $d_n(w)=i_{n-1}(v)\in\widehat{\C}_{n-1}^{\Un(1)}$. Writing $w$ as a sum of elements $v_0\in\widehat{\C}_n^{\Un(1)}$ and $w_0\in\ker(\hat{P}_n)\subset\widehat{\C}_n$ in line with  \eqref{decomp} implies consequently $\rho(\lambda)d_{n}(w)=d_n(v_0)+\rho(\lambda)(d_n(w_0))=d_n(v_0)+d_n(w_0)=d_n(w)$ for every element $\lambda\in\Un(1)$. Consequently, $d_n(w_0)\in\widehat{\C}_{\bullet}^{\Un(1)}$. Hence, $d_n(w_0)=0$. We can write $w$ as a sum 
\begin{equation}\label{decompofw}w=v_0+w_0,\end{equation}  
where $v_0\in\widehat{\C}_n^{\Un(1)}$ and $z\in\widehat{Z}_n\cap\ker(\hat{P}_n)$ where $\widehat{Z}_n$ is the space of Hochschild $n$-cycles in $\widehat{\C}_n$.

Let $E=\epsilon\sum_{k=1}^{\dim\mathfrak{h}}y_k\frac{\partial}{\partial y_k}$ be the Euler field on $\mathfrak{h}$ corresponding to an element $\epsilon$ from the Lie algebra $\mathfrak{u}(1)$. Then, the induced action of $\mf{u}(1)$ on the differential-graded algebra $\widehat{\C}_{\bullet}$ is given by   
\begin{equation}
\label{la-action}
 \rho_*(\epsilon)(a_0\otimes\dots\otimes a_n):=\frac{d}{dt}\big|_{t=0}\rho(e^{\epsilon t})(a_0\otimes\dots\otimes a_n)=\sum_{j=0}^na_0\otimes\dots\otimes[E, a_j]\otimes\dots\otimes a_n,
\end{equation}
for every $\epsilon\in\mathfrak{u}(1)\cong i\bb{R}$ and for every $a_0\otimes\dots\otimes a_n\in\widehat{\C}_n$. We notice that the free $\Un(1)$-action $\rho$ on $\ker(\hat{P}_{\bullet})$ induces a free $\mathfrak{u}(1)$-action $\rho_*$ on $\ker(\hat{P}_{\bullet})$. In what follows, we show that the image $\rho_*(i)$ of the sole generator of $\mathfrak{u}(1)$ is a null-homotopic endomorphism of the chain complex $\widehat{\C}_{\bullet}$. 
For that purpose we define a $\bb{C}$-linear map $h:\widehat{\C}_{n}\longrightarrow\widehat{\C}_{n+1}$ by
\begin{equation*}
%\label{retraction}
a_0\otimes\dots\otimes a_n\mapsto \sum_{j=0}^n(-1)^{j+1}a_0\otimes\dots\otimes a_j\otimes E\otimes a_{j+1}\otimes\dots\otimes a_n,
\end{equation*}
where here $E$ stands for the Euler vector field on $\mathfrak{h}$ associated to the generator $i$ of $\mf{u}(1)$. Let $a_0\otimes\dots\otimes a_n\in\widehat{\C}_n$. Then we compute
\begin{align}
\label{dh}
&\mathsmaller{dh(a_0\otimes\dots\otimes a_n)}=\mathsmaller{d\big( \sum_{j=0}^n(-1)^{j+1}a_0\otimes\dots\otimes a_j\otimes E\otimes a_{j+1}\otimes\dots\otimes a_n\big)}~~~~~~~~~~~~~~~~~~~~~~~~\nonumber\\
&\mathsmaller{=\sum_{j=0}^{n}\big( a_0\otimes\dots\otimes Ea_{j+1}\otimes\dots a_n -a_0\otimes\dots\otimes a_jE\otimes\dots\otimes a_n\big)}\nonumber\\
&\mathsmaller{+\sum_{j=0}^{n}\sum_{k=0}^{j-1}(-1)^{j+k+1}a_0\otimes\dots\otimes a_ka_{k+1}\otimes\dots\otimes(a_j\otimes E\otimes a_{j+1})\otimes\dots\otimes a_n}\nonumber\\
&\mathsmaller{-\sum_{j=0}^{n}\sum_{k=j+1}^n(-1)^{j+k+1}a_0\otimes\dots\otimes(a_j\otimes E\otimes a_{j+1})\otimes\dots\otimes a_ka_{k+1}\otimes\dots\otimes a_n}
\end{align}
as well as 
\begin{align}
\label{hd}
&\mathsmaller{hd(a_0\otimes\dots\otimes a_n)}=\mathsmaller{h(\sum_{k=0}^n(-1)^{k}a_0\otimes\dots\otimes a_ka_{k+1}\otimes\dots\otimes a_n)}~~~~~~~~~~~~~~~~~~~~~~~~~~~~~~~~~~~~~~~~~\nonumber\\
&\mathsmaller{=-\sum_{j=0}^{n}\sum_{k=0}^{j-1}(-1)^{j+k+1}a_0\otimes\dots\otimes a_ka_{k+1}\otimes\dots\otimes a_j\otimes E\otimes a_{j+1}\otimes\dots\otimes a_n}\nonumber\\
&\mathsmaller{+\sum_{j=0}^n\sum_{k=j+1}^n(-1)^{j+k+1}a_0\otimes\dots\otimes a_j\otimes E\otimes a_{j+1}\otimes\dots\otimes a_ka_{k+1}\otimes\dots\otimes a_n}.
\end{align}
The sum of \eqref{dh} and \eqref{hd} yields
\begin{align}
\label{nullhom}
\mathsmaller{(dh+hd)(a_0\otimes\dots\otimes a_n)}&\mathsmaller{=\sum_{j=0}^{n}(a_0\otimes\dots\otimes Ea_{j+1}\otimes\dots a_n-a_0\otimes\dots\otimes a_jE\otimes\dots\otimes a_n)}\nonumber\\
&\mathsmaller{=\sum_{j=1}^{n}a_0\otimes\dots\otimes Ea_{j}\otimes\dots\otimes a_n-\sum_{j=1}^na_0\otimes\dots\otimes a_jE\otimes\dots\otimes a_n }\nonumber\\
&\mathsmaller{+[E, a_0]\otimes\dots\otimes a_n}\nonumber\\
&\mathsmaller{=\sum_{j=0}^na_0\otimes\dots\otimes[E, a_j]\otimes\dots\otimes a_n},
\end{align}
which by Definition \eqref{la-action} is equal to $\rho_*(i)(a_0\otimes\dots\otimes a_n)$. Hence, $h$ is a contraction  and $\rho_{*}(i)$ is null-homotopic. Since $\rho_*$ is a free $\mathfrak{u}(1)$-action on $\ker(\hat{P}_{\bullet})$, the endomorphism $\rho_*(i)$ is an injective chain complex map of $\widehat{\C}_{\bullet}$ which is invertible on $\Ima\rho_*(i)$. Take an $n$-cycle $z\in\ker(\hat{P}_n)\cap\widehat{Z}_n$. Then, by computation \eqref{nullhom}, we have $\rho_{*}(i)(z)= dh(z)$. As $\rho_{*}(i)$ is invertible   on $\Ima\rho_*(i)$, we have $z=\rho_*(i)^{-1}dh(z)=d\rho_*(i)^{-1}(hz)\in\ker(\hat{P}_n)\cap\widehat{B}_n$. Consequently, if $(v, w)\in\ker(D_n)$, then decomposition \eqref{decompofw} can be refined as 
\begin{equation*}
w= v_0+w_0
\end{equation*}
where $v_0\in\widehat{\C}_n^{\Un(1)}$ and $w_0\in\widehat{B}_n\cap\ker(\hat{P}_n)$. Ergo, there are always $-x\in\widehat{\C}_n^{\Un(1)}$ and $y\in\widehat{\C}_{n+1}$ such that $w=d_{n+1}(y)-i_{n}(x)$. Since $d_n(w)=-d_n(i_n(x))=-i_{n-1}(d_n(x))=-i_{n-1}(v)$, then the injectivity of $i_{n-1}$ stipulates that $v=-d_n(x)$.
Consequently, $(v, w)$ can be written as $(v, w)=(-d_n(x), d_{n+1}(y)-i_n(x))$ for an appropriate $x\in\widehat{\C}_{n}^{\Un(1)}$ and $y\in\widehat{\C}_n$. This implies $\ker(D_n)\subseteq\Ima(D_{n+1})$. Hence, $i$ is a quasi-isomorphism. In an analogous manner we demonstrate that the inclusion morphism $i: \C_{\bullet}^{\Un(1)}\hookrightarrow \C_{\bullet}$ is a quasi-isomorphism.

Now, we show that $\widehat{\C}_{\bullet}^{\Un(1)}=\C_{\bullet}^{\Un(1)}$. Indeed, as the $\Un(1)$-fixed point subspaces of $\widehat{\mc{D}}(\mathfrak{h})\rtimes G$ and $\mc{D}_{\textrm{alg.}}(\mf{h})\rtimes G$ are generated by $\bb{C}G$ and the subspace spanned by $\{y_i\pd{y_j}~|~ i, j=1, \dots, \dim\mathfrak{h}\}$, they are identical. Both quasi-isomorphisms $i: \widehat{\C}_{\bullet}^{\Un(1)}\to\widehat{\C}_{\bullet}$ and $i: \C_{\bullet}^{\Un(1)}\to\C_{\bullet}$ are related by the zig-zag diagram
\begin{equation*}
\widehat{\C}_{\bullet}\xhookleftarrow{~~~~i~~~~}\widehat{\C}_{\bullet}^{\Un(1)}=\C_{\bullet}^{\Un(1)}\xhookrightarrow{~~~~i~~~~} \C_{\bullet}
\end{equation*}
which at the level of homologies provides the statement of the proposition.
\end{proof}
The first direct consequence of Proposition \ref{prophh} is the following result.  
\begin{corollary}
Let $U$ be a simply connected $G$-invariant affine or Stein open subset of a complex vector space $\mf{h}$ and let $\mc{D}(U)$ be the algebra of holomorphic differential operators on $U$.  Then, \[\hh_j(\mc{D}(U)\rtimes G)\cong\hh_j(\mc{D}_{\textrm{alg.}}(\mf{h})\rtimes G)\]
as complex vector spaces.
\end{corollary}
\begin{proof}
Let us consider the injective maps of chain complexes $i_1: \C_{\bullet}(\mc{D}_{\textrm{alg.}}(\mf{h})\rtimes G)\hookrightarrow \C_{\bullet}(\mc{D}(U)\rtimes G)$ and $i_2:\C_{\bullet}(\mc{D}(U)\rtimes G)\hookrightarrow\widehat{\C}_{\bullet}(\widehat{\mc{D}}(\mf{h})\rtimes G)$. The composition $I:=i_2\circ i_1$ fits in the following commutative diagram 
\begin{equation*}
\begin{tikzcd}
\C_{\bullet}(\mc{D}_{\textrm{alg.}}(\mf{h})\rtimes G)^{\Un(1)}\arrow[r, "="]\arrow[d, hook, left, "\textrm{quasi-iso.}"]& \widehat{\C}_{\bullet}(\widehat{\mc{D}}(\mf{h})\rtimes G)^{\Un(1)}\arrow[d, hook, "\textrm{quasi-iso.}"] \\
\C_{\bullet}(\mc{D}_{\textrm{alg.}}(\mf{h})\rtimes G) \arrow[r, hook, "I"] & \widehat{\C}_{\bullet}(\widehat{\mc{D}}(\mf{h})\rtimes G)
\end{tikzcd}
\end{equation*}
where the upper horizontal map and the vertical maps are  the identity, respectively the injective quasi-isomorphisms from the proof of Proposition \ref{prophh}. At the level of Hochschild homology, $I_{*}$ is an isomorphism whence $i_{2*}$ is surjective. Since $i_{2*}$ is clearly injective, too, the claim follows.    
\end{proof}
\begin{remark}
\label{hhofholdiffops}
$i)$ If $U$ is not simply connected, morphism $i_2$ might not exist. Consider, for example, the fixed-point-free subset $U=\mf{h}_{\textrm{reg}}$. 

$ii)$ The above result remains valid for disjoint unions of connected $G$-invariant open sets $U_i$ in $G$-spaces $\mf{h}_i$. Namely, it can be shown that $\hh_j(\mc{D}(\coprod_iU_i)\rtimes G)\cong\hh_j(\mc{D}_{\textrm{alg.}}(\coprod_i\mf{h}_i)\rtimes G)$.  
\end{remark}
From now on, unless stated otherwise, we denote by $\hbar$ multiple formal indeterminates $\hbar_1, \dots, \hbar_k$ with $k$ equal to the number of conjugacy classes of complex reflections in $G$ and we set $\bb C[\hbar]:=\bb C[\hbar_1, \dots, \hbar_k]$ and $\bb C\llbracket\hbar\rrbracket:=\bb C\llbracket\hbar_1, \dots, \hbar_k\rrbracket$, respectively. We only make an exception of this nomenclature for the field of Laurent series $\bb C\lau{\hbar}$ in a single indeterminate $\hbar$ which for brevity is denoted by $\bb K$ throughout the text. Accordingly, we denote by $\widehat{H}_{1, \hbar}(\mathfrak{h}, G)$ the $k$-parameter formal deformation of $\widehat{\mc{D}}(\mf{h})\rtimes G$. 
%In that case the formal parameter $\hbar$ is a central element in the formal Cherednik algebra $\widehat{H}_{1, \hbar}(\mathfrak{h}, G)$. Hence, one can localize the underlying non-commutative ring with respect to $\hbar$ the same way as any other commutative ring. 
We use the notation $\Chered{G}$ for the localization of the $1$-parameter formal deformation $\widehat{H}_{1, \hbar c}(\mathfrak{h}, G)$ with respect to the multiplicative submonoid $T:=\{\hbar^{n}~|~n\in\bb{Z}_{>0}\}$ where $\hbar c$ stands for  $\hbar c_1, \dots, \hbar c_k$ with $\hbar$ a formal indeterminate and $c_1, \dots, c_k$ fixed complex-valued parameters. In particular, we have $\Chered{G}=\widehat{H}_{1, \hbar c}(\mathfrak{h}, G)[T^{-1}]=\widehat{H}_{1, \hbar c}(\mathfrak{h}, G)\otimes_{\bb{C}\llbracket\hbar\rrbracket}\bb K$. 
%Similarly, $\mc{D}_{\textrm{alg.}}(\mathfrak{h})\rtimes G\llbracket\hbar\rrbracket$ and $\mc{D}_{\textrm{alg.}}(\mathfrak{h})\rtimes G\lau{\hbar}$ are isomorphic to $\Weyl_{\hbar}(\mf{h}\oplus\mf{h}^*)\rtimes G$ and its localization with respect to $T$, respectively, where $\Weyl_{\hbar}(\mf{h}\oplus\mf{h}^*)$ denotes the homogenization of the Weyl algebra $\Weyl(\mf{h}\oplus\mf{h}^*)$ by means of $\hbar$ with $\deg(\hbar)=1$. 
Let $\bb C[\varepsilon]$ with $\varepsilon^2=0$ be the ring of dual numbers. The following proposition is a consequence of Theorem \ref{hhcher} and Proposition \ref{prophh} and the proof of its first half mimics and expounds the proof of \cite[Proposition 1]{RT12}. 
\begin{proposition}
\label{hhaj}
$i)$ There is an isomorphism of $\bb C\llbracket\hbar\rrbracket$-modules 
\[\hh_{j}\big(\widehat{H}_{1, \hbar}(\mathfrak{h}, G)\big)\cong\hh_j\big(\mc{D}_{\textrm{alg.}}(\mf{h})\rtimes G\big)\llbracket\hbar\rrbracket\cong\bb C\llbracket\hbar\rrbracket^{a_j}\] 
where $a_j$ denotes the number of conjugacy classes of elements of $G$ having eigenvalue $1$ with multiplicity $j$.

$ii)$ There is an isomorphism of $\bb K$-vector spaces 
\[\hc_{\bullet}\big(\widehat{H}_{1, \lau{\hbar}}(\mf{h}, G)\otimes_{\bb C}\bb C[\varepsilon]\big)\cong\hc_{\bullet}\big(\mc{D}_{\textrm{alg.}}(\mf{h})\rtimes G\lau{\hbar}\otimes_{\bb C}\bb C[\varepsilon]\big).\]
\end{proposition}
\begin{proof}
%=============DELETED ON 17.06.2021====================
%\quad The $\bb{C}\llbracket\hbar\rrbracket$-algebra $\bb K$ of formal Laurent series has an increasing and exhaustive ring filtration by $\bb{C}\llbracket\hbar\rrbracket$-modules $\bb{C}\llbracket\hbar\rrbracket\subset\bb{C}\llbracket\hbar\rrbracket\hbar^{-1}\subset\dots\subset\bb{C}\llbracket\hbar\rrbracket\hbar^{-p}\subset\dots\subset\bb K$, which gives rise to a directed system in the category of $\bb{C}\llbracket\hbar\rrbracket$-modules. Since $\bb K$ satisfies the same universal property as the direct limit, by uniqueness of the universal property of $\colim_p\bb{C}\llbracket\hbar\rrbracket\hbar^{-p}$, there is an isomorphism of $\bb{C}\llbracket\hbar\rrbracket$-modules $\bb K\cong\colim_{p}\bb{C}\llbracket\hbar\rrbracket\hbar^{-p}$. Therefore, if $\widehat{K}_{n}:=\widehat{K}_n(\widehat{H}_{1, \hbar}(\mathfrak{h}, G))$ designates the $n$-th Hochschild chain complex $\bb{C}\llbracket\hbar\rrbracket$-module of $\widehat{H}_{1, \hbar}(\mathfrak{h}, G)$, then we recast the $n$-th Hochschild chain complex $\bb K$-module  $\widehat{\C}_{n}(\Chered{G})$ in the form
%\begin{align*}
%\widehat{\C}_{n}(\Chered{G})\cong\colim_{p}\widehat{K}_n(\widehat{H}_{1, \hbar}(\mathfrak{h}, G))\hbar^{-p}.
%\end{align*}
%As directed limits over directed systems commute with the homology functor, we get for the Hochschild homology of $\Chered{G}$,
%\begin{align}
%\label{lochh}
%\hh_{\bullet}(\Chered{G})=\colim_{p}\hh_{\bullet}(\widehat{H}_{1, \hbar}(\mathfrak{h}, G))\hbar^{-p}.
%\end{align}
%=============DELETED ON 17.06.2021====================
$i)$ To simplify the notation, let us denote by $(\hbar)$ the maximal ideal $(\hbar_1, \dots, \hbar_k)$ in $\bb C\llbracket\hbar\rrbracket$. To compute $\hh_{\bullet}(\widehat{H}_{1, \hbar}(\mathfrak{h}, G))$, we define an increasing, exhaustive and bounded above $\bb{C}\llbracket\hbar\rrbracket$-module filtration on $\widehat{K}_{\bullet}:=\widehat{K}_{\bullet}(\widehat{H}_{1, \hbar}(\mathfrak{h}, G))$ by $F_{p}\widehat{K}_{\bullet}=\widehat{K}_{\bullet}(\hbar)^{-p}$ for $p\in\bb{Z}_{\leq0}$. Then the fact that  
 \begin{align*}
  \widehat{K}_{n}&=\Bigg(\frac{\bb{C}\llbracket\mathfrak{h}\rrbracket\otimes_{\bb{C}[\mathfrak{h}]}T^{\bullet}(\mathfrak{h}\oplus\mathfrak{h}^{\ast})\otimes_{\bb{C}}\bb{C}G[\hbar]}{J}\Bigg)^{\hat\otimes_{\bb{C}}\mathlarger{n+1}}\otimes_{\bb{C}[\hbar]}\bb{C}\llbracket\hbar\rrbracket
 \end{align*}
combined with $\bb{C}\llbracket\hbar\rrbracket/\hbar^p\bb{C}\llbracket\hbar\rrbracket\cong\bb{C}[\hbar]/(\hbar)^p\bb{C}[\hbar]$ implies that the chain complex is complete in the $(\hbar)$-adic topology. Hence, the filtration on $\widehat{K}_{\bullet}$ is complete. The filtration of the chain complex $\widehat{K}_{\bullet}$ determines by the \emph{construction theorem} a spectral sequence $\{E_{pq}^r\}$ starting with 
\begin{align}
\label{Enull}
E_{pq}^0&=F_{p}\widehat{K}_{p+q}/F_{p-1}\widehat{K}_{p+q}\nonumber\\
&\cong(\hbar)^{-p}\widehat{\mc{D}}(\mathfrak{h})\rtimes G\hat\otimes_{\bb{C}}\dots\hat\otimes_{\bb{C}}\widehat{\mc{D}}(\mathfrak{h})\rtimes G\nonumber\\
&=\begin{cases}(\hbar)^{-p}\widehat{\C}_{p+q}(\widehat{\mc{D}}(\mathfrak{h})\rtimes G)~~~~~\textrm{for}~~p\in\bb{Z}_{\leq0}\\ 0~~~~~~~~~~~~~~~~~~\textrm{for}~~p\in\bb{Z}_{>0},\end{cases}
\end{align}
and consequently
\begin{align}
\label{Eone}
E_{pq}^1&=\h_{p+q}(E_{p\ast}^0)\nonumber\\
&=\begin{cases}(\hbar)^{-p}\bb{C}^{a_{p+q}}~~~~\textrm{for}~~p\in\bb{Z}_{\leq0} \\0~~~~~~~~~~~~~~~~~~\textrm{for}~~p\in\bb{Z}_{>0},\end{cases}
\end{align}
where the last line follows from Proposition \ref{prophh}. By  \eqref{Enull}, the entries $E_{pq}^0=0$ for every $p>0$. Hence, the spectral sequence $\{E_{pq}^r\}$ is bounded from above. 

Generically, the space $\mathfrak{h}$ is a semisimple left $\bb CG$-module which by Maschke's theorem can be decomposed in a direct sum $\mathfrak{h}=\mathfrak{h}_1^{\oplus m_1}\oplus\dots\mathfrak{h}_r^{\oplus m_r}$ of isotypic components of simple left $\bb CG$-submodules $\mathfrak{h}_{j}$ of $\mathfrak{h}$ with multiplicity  $m_j$, $j=1,\dots, r$. One can correspondingly express $\mc{D}_{\textrm{alg.}}(\mathfrak{h})\rtimes G$ as
\begin{equation*}
\mc{D}_{\textrm{alg.}}(\mf{h})\rtimes G\cong \big(\mc{D}_{\textrm{alg.}}(\mathfrak{h}_0)\rtimes G\big)^{\otimes_{\bb{C}}\mathlarger{m_0}}\otimes_{\bb C}\dots\otimes_{\bb C}\big(\mc{D}_{\textrm{alg.}}(\mathfrak{h}_r)\rtimes G\big)^{\otimes_{\bb{C}}\mathlarger{m_r}}.
\end{equation*}
Finally, K\"unneth's formula for chain complexes yields the isomorphism
\begin{align}
\label{kuenneth}
&\mathsmaller{\hh_{p+q}(\mc{D}_{\textrm{alg.}}(\mathfrak{h})\rtimes G)}
\cong\underset{ \overset{r}{\underset{j=0}{\sum}}\overset{m_j}{\underset{s_j=1}{\sum}} \alpha_{s_j}^j=p+q}{\bigoplus}\Bigg\{\underset{\substack{j=0,\dots, r\\ s_j=1, \dots, m_j}}{\bigotimes}\hh_{\mathlarger{\alpha}_{s_j}^j}(\mc{D}_{\textrm{alg.}}(\mathfrak{h}_j)\rtimes G)\Bigg\}.
\end{align}
According to \cite[Formula (2.12)]{EG02}, for every simple left $\bb CG$-submodule $\mathfrak{h}_j$, we have $a_{\alpha_{s_j}^j}=0$ if $\alpha_{s_j}^j$ is odd. Consequently, $\hh_{\alpha_{s_j}^j}(\mc{D}_{\textrm{alg.}}(\mathfrak{h}_j)\rtimes G)=0$ when $\alpha_{s_j}^j$ is odd. As whenever $p+q$ is odd, every summand on the right hand side of Isomorphism \eqref{kuenneth} has at least one odd index $\alpha_{s_j}^j$, it follows that $\hh_{p+q}(\mc{D}_{\textrm{alg.}}(\mathfrak{h})\rtimes G)=0$ for $p+q$ odd. Thus, $E_{pq}^1=E_{pq}^{\infty}$, which renders the sequence regular. The complete convergence theorem implies that $E_{pq}^r$ converges to $\hh_{p+q}(\widehat{H}_{1, \hbar}(\mathfrak{h}, G))$, that is,
\begin{align}
\label{conv}
E_{pq}^1&\cong\frac{F_{p}\hh_{p+q}(\widehat{H}_{1, \hbar}(\mathfrak{h}, G))}{F_{p-1}\hh_{p+q}(\widehat{H}_{1, \hbar}(\mathfrak{h}, G))}\nonumber\\
&=(\hbar)^{-p}\hh_{p+q}(\widehat{H}_{1, \hbar}(\mathfrak{h}, G)).
\end{align}
Equating \eqref{Eone} with \eqref{conv} yields $\hh_{p+q}(\widehat{H}_{1, \hbar}(\mathfrak{h}, G))= \bb{C}\llbracket\hbar\rrbracket^{a_{p+q}}$. %Claim \emph{i)} of the proposition follows immediately from \eqref{lochh}.\\

$ii)$ 
As discussed in Section \ref{rational}, the algebra $\widehat{H}_{1, \lau{\hbar}}(\mf{h}, G)$ has an increasing filtration which is inherited by the spherical subalgebra $\widehat{B}_{1, \lau{\hbar}}(\mf{h}, G):={\bf e}\widehat{H}_{1, \lau{\hbar}}(\mf{h}, G){\bf e}$, where ${\bf e}={\bf e}^2$ is the idempotent in $G$. By \cite{EG02}, these algebras are Morita equivalent. Hence, due to the Morita invariance of the cyclic homology we have
\begin{align}
\label{moritaiso}
\hc_{\bullet}(\widehat{H}_{1, \lau{\hbar}}(\mf{h}, G)\otimes\bb C[\varepsilon])\cong\hc_{\bullet}(\widehat{B}_{1, \lau{\hbar}}(\mf{h}, G)\otimes\bb C[\varepsilon]). 
 \end{align}
 On the other hand, the associated graded algebra of $\widehat{B}_{1, \lau{\hbar}}(\mf{h}, G)\otimes\bb C[\varepsilon]$ is isomorphic to the Poisson algebra $\Sym^{\bullet}(\mf{h})^G\otimes\bb C[\varepsilon]$. According to Section $5$ of \cite{Bry88}, there is a mixed bicomplex $(\Omega_{\Sym^{\bullet}(\mf{h})^G\otimes\bb C[\varepsilon]}^{\bullet}, \delta, d)$, where $\delta$ is the Brylinski differential of degree $-1$ and $d$ is the de Rham differential. It follows by \cite[Theorem 2]{Kas88} that 
 \begin{align*}
 \hc_{\bullet}(\widehat{B}_{1, \lau{\hbar}}(\mf{h}, G)\otimes\bb C[\varepsilon])\cong\hc_{\bullet}((\Omega_{\Sym^{\bullet}(\mf{h})^G\lau{\hbar}\otimes\bb C[\varepsilon]}^{\bullet}, \delta, d)).
 \end{align*} 
 This isomorphism coupled with Isomorphism \eqref{moritaiso} yields
 \begin{align*}
\hc_{\bullet}(\widehat{H}_{1, \lau{\hbar}}(\mf{h}, G)\otimes\bb C[\varepsilon])\cong\hc_{\bullet}((\Omega_{\Sym^{\bullet}(\mf{h})^G\lau{\hbar}\otimes\bb C[\varepsilon]}^{\bullet}, \delta, d)).  
  \end{align*}
After we repeat the above exactly the same way for $\widehat{\mc{D}}(\mf{h})\rtimes G\lau{\hbar}$, we obtain
\begin{align*}
\hc_{\bullet}(\widehat{\mc{D}}(\mf{h})\rtimes G\lau{\hbar}\otimes\bb C[\varepsilon])\cong\hc_{\bullet}((\Omega_{\Sym^{\bullet}(\mf{h})^G\lau{\hbar}\otimes\bb C[\varepsilon]}^{\bullet}, \delta, d)).  
 \end{align*}
 The statement follows immediately.    
\end{proof}
The space of traces on $\widehat{H}_{1, \hbar}(\mathfrak{h}, G)$ is  isomorphic to the zeroth cyclic cohomology group of $\widehat{H}_{1, \hbar}(\mathfrak{h}, G)$ which in turn is isomorphic to the dual of the zeroth Hochschild homology $\bb C\llbracket\hbar\rrbracket$-module $\hh_0(\widehat{H}_{1, \hbar}(\mathfrak{h}, G))$. As by Proposition \ref{hhaj}, $i)$, the zeroth Hochschild homology of $\widehat{H}_{1, \hbar}(\mathfrak{h}, G)$ is a free $\bb C\llbracket\hbar\rrbracket$-module of finite rank, the space of $\bb C\llbracket\hbar\rrbracket$-linear traces on $\widehat{H}_{1, \hbar}(\mathfrak{h}, G)$ is isomorphic to the zeroth Hochschild homology $\bb C\llbracket\hbar\rrbracket$-module $\hh_0(\widehat{H}_{1, \hbar}(\mathfrak{h}, G))$. The ensuing corollary ensures that there are circumstances in which $\widehat{H}_{1, \hbar}(\mathfrak{h}, G)$ possesses nontrivial traces. 
 \begin{corollary}
 \label{dimoftracegroupofformalcher}
Suppose $G\subset\GL(\mathfrak{h})$ is a well-generated complex reflection group with $\mathfrak{h}^G=\{0\}$. Then, the rank of~~$\hh_0(\widehat{H}_{1, \hbar}(\mathfrak{h}, G))$ over $\bb C\llbracket\hbar\rrbracket$ is at least $1$.
 \end{corollary}
 \begin{proof}
Theorem \ref{thm1.1.1} along with condition $\mathfrak{h}^G=\{0\}$ and Lemma \ref{supp(G)} imply that $\supp(G)=\mathfrak{h}_1\oplus\dots\oplus\mathfrak{h}_m=\Span_{S}\{\alpha_s^{\vee}\}$, where $\mathfrak{h}_i$ is an irreducible $G_i$-module for every $i=1, \dots, m$ and $\alpha_s^{\vee}$ is the root of the complex reflection $s$. Since $G$ is well-generated, we have that 
$\mathfrak{h}_1\oplus\dots\oplus\mathfrak{h}_m=\bigoplus_{i=1}^m\bigoplus_{s_i\in \mc{S}_i}\Span\{\alpha_{s_i}^{\vee}\}$. Since from Theorem \ref{thm1.1.1} we know that each irreducible complex reflection subgroups $G_i$ is generated by those complex reflections in $G$ whose roots belong to $\mathfrak{h}_i$, it follows that each $G_i$ for $1=1, \dots, m$, is an irreducible, well-generated complex reflection group. Consequently, by Lemma \ref{coxeter} each $G_i$ possesses a Coxeter element $c_i$. Take $c:=(c_1, \dots, c_m)$. By Lemma \ref{nouniteigenvalue}, this group element in $G_1\times\dots\times G_m$ has no eigenvalue equal to $1$ and corresponds to an element in $G$ with no eigenvalue equal to $1$. Thus, $a_0\geq1$ in this case. The claim follows then by Proposition \ref{hhaj}, $i)$.
\end{proof}
Let $\mc{A}_{n-l, l}^H:=\widehat{\mc{D}}_{n-l}\hat{\otimes}\widehat{H}_{1, c}(\bb C^l, H)$. Let $\mc{A}_{n-l, l}^{H, \hbar}$ be the $\bb C\llbracket\hbar\rrbracket$-algebra $\widehat{\mc{D}}_{n-l}^{\hbar}\hat{\otimes}_{\bb C\llbracket\hbar\rrbracket}\widehat{H}_{1, \hbar}(\bb C^l, H)$ where $\widehat{\mc{D}}_{n-l}^{\hbar}$ is the algebra of differential operators in $n-l$ variables with coefficients in the ring $\bb C\llbracket x_1, \dots, x_{n-l}, \hbar\rrbracket$. 
Let $\mc{A}_{n-l, l}^{H, \lau{\hbar}}$ denote the $\bb K$-algebra $\widehat{\mc{D}}_{n-l}^{\lau{\hbar}}\hat{\otimes}_{\bb K}\widehat{H}_{1, \lau{\hbar}}(\bb C^l, H)$ where $\widehat{\mc{D}}_{n-l}^{\lau{\hbar}}$ is the algebra of differential operators on the formal neighborhood of zero in $\bb K^{n-l}$. With the help of the previous results, we compute the Hochschild homology of $\mc{A}_{n-l, l}^{H, \lau{\hbar}}$ and the cyclic homology of the $\bb Z_2$-graded algebra $\mc{A}_{n-l, l}^{H, \lau{\hbar}}\otimes_{\bb C}\bb C[\varepsilon]$ which we need later for the proof of Proposition \ref{cwhom1}.  
\begin{corollary}
\label{hhchofhcmodule}
For every $m\in\bb Z_{\geq0}$, 
\[i)\quad\hh_m\big(\mc{A}_{n-l, l}^{H, \lau{\hbar}}\big)\cong\hh_{m-2n+2l}\big(\mc{D}_{\textrm{alg.}}(\bb C^l)\rtimes H\lau{\hbar}\big)\cong\bb K^{a_{m-2n+2l}}\]
where $a_{m-2n+2l}$ is as in Proposition \ref{hhaj}, \emph{i)}.
\begin{align*}
&ii)\quad\hc_m\big(\mc{A}_{n-l, l}^{H, \lau{\hbar}}\otimes_{\bb C}\bb C[\varepsilon]\big)\cong\hc_{m-(2n-2l)}\big(\mc{D}_{\textrm{alg.}}(\bb C^l)\rtimes H\lau{\hbar}\otimes_{\bb C}\bb C[\varepsilon]\big)\\
&\hspace{12.4em}\cong\bigoplus_{\gamma\in\Conj(H)}\big(\hc_{m-(2n-2l)-2k_{\gamma}}(\bb K)\oplus\bb K\cdot[\hspace{-1.9em}\underbrace{\varepsilon\otimes\dots\otimes\varepsilon}_{m-(2n-2l)-2k_{\gamma}-\textrm{times}}\hspace{-1.9em}]\big)
\end{align*}
where $2k_{\gamma}:=\dim(\bb C^l\oplus\bb C^{l*})^{\gamma}$, $\cdot$ denotes multiplication and $[\varepsilon\otimes\dots\otimes\varepsilon]$ is the cohomology class of  $\varepsilon\otimes\dots\otimes\varepsilon$.
\end{corollary}
\begin{proof}
$i)$~~This follows from K\"unneth's formula, Proposition \ref{prophh}, \cite[Theorem 2]{Wod87} and Proposition \ref{hhaj}, $i)$ together with the following fact. The $\bb{C}\llbracket\hbar\rrbracket$-algebra $\bb K$ of formal Laurent series has an increasing and exhaustive ring filtration by $\bb{C}\llbracket\hbar\rrbracket$-modules $\bb{C}\llbracket\hbar\rrbracket\subset\bb{C}\llbracket\hbar\rrbracket\hbar^{-1}\subset\dots\subset\bb{C}\llbracket\hbar\rrbracket\hbar^{-p}\subset\dots\subset\bb K$, where $\hbar$ is a single indeterminate here. This gives rise to a directed system in the category of $\bb{C}\llbracket\hbar\rrbracket$-modules. Since $\bb K$ satisfies the same universal property as the direct limit, 
%by uniqueness of the universal property of $\colim_p\bb{C}\llbracket\hbar\rrbracket\hbar^{-p}$, 
there is an isomorphism of $\bb{C}\llbracket\hbar\rrbracket$-modules $\bb K\cong\colim_{p}\bb{C}\llbracket\hbar\rrbracket\hbar^{-p}$. Therefore, if $\widehat{K}_{n}:=\widehat{K}_n(\widehat{H}_{1, \hbar c}(\mathfrak{h}, G))$ denotes the $n$-th Hochschild chain complex module of the $1$-parameter formal deformation $\widehat{H}_{1, \hbar c}(\mathfrak{h}, G)$, then we recast the $n$-th Hochschild chain complex module $\widehat{\C}_{n}(\Chered{G})$ in the form $\widehat{\C}_{n}(\Chered{G})\cong\colim_{p}\widehat{K}_n(\widehat{H}_{1, \hbar c}(\mathfrak{h}, G))\hbar^{-p}$. As directed limits over directed systems commute with the homology functor, we get for the Hochschild homology $\hh_{\bullet}(\Chered{G})=\colim_{p}\hh_{\bullet}(\widehat{H}_{1, \hbar c}(\mathfrak{h}, G))\hbar^{-p}$
%\end{align*}
which combined with Proposition \ref{hhaj}, $i)$ implies the result. 

$ii)$~~Without any loss of generality assume that $\deg(\varepsilon)=0$. Since the ring of dual numbers is a flat $\bb C$-module, the tensor product over $\bb C$ with $\bb C[\varepsilon]$ preserves the natural injection $\widehat{H}_{1, \lau{\hbar}}(\bb C^l, H)\hookrightarrow\mc{A}_{n-l, l}^{H, \lau{\hbar}}$ and we obtain an injective embedding of $\bb Z_2$-graded algebras $f: \widehat{H}_{1, \lau{\hbar}}(\bb C^l, H)\otimes\bb C[\varepsilon]\hookrightarrow\mc{A}_{n-l, l}^{H, \lau{\hbar}}\otimes\bb C[\varepsilon]$. In turn, it induces an injection of chain complexes $f: \widehat{\C}_{\bullet}\big(\widehat{H}_{1, \lau{\hbar}}(\bb C^l, H)\otimes\bb C[\varepsilon]\big)\hookrightarrow\widehat{\C}_{\bullet+(2n-2l)}\big(\mc{A}_{n-l, l}^{H, \lau{\hbar}}\otimes\bb C[\varepsilon]\big)$ by
 \begin{equation*}
\hspace{0.3em}{\big(a_0\otimes x_0+y_0\varepsilon\big)\otimes\dots\otimes \big(a_m\otimes x_m+y_m\varepsilon\big)}\mapsto\hspace{-0.3em}
1^{\otimes(2n-2l)}\otimes\big(a_0\otimes x_0+y_0\varepsilon\big)\otimes\dots\otimes\big(a_m\otimes x_m+y_m\varepsilon\big)
\end{equation*}  
which by abuse of notation we keep calling $f$. The corresponding induced map \[f_*: \hh_{\bullet}\big(\widehat{H}_{1, \lau{\hbar}}(\bb C^l, H)\otimes\bb C[\varepsilon]\big)\rightarrow\hh_{\bullet+(2n-2l)}\big(\mc{A}_{n-l, l}^{H, \lau{\hbar}}\otimes\bb C[\varepsilon]\big)\] is obviously injective, too. Indeed, assume that $f_*\Big(\big(a_0\otimes x_0+y_0\varepsilon\big)\otimes\dots\otimes \big(a_m\otimes x_m+y_m\varepsilon\big)\Big)$ is an $m+(2n-2l)$-boundary. This implies that there are $c_0, \dots, c_{m+(2n-2l)+1}\in\mc{A}_{n-l, l}^{H, \lau{\hbar}}\otimes\bb C[\varepsilon]$ such that
\begin{align*}
&f_*\Big(\big(a_0\otimes x_0+y_0\varepsilon\big)\otimes\dots\otimes \big(a_m\otimes x_m+y_m\varepsilon\big)\Big)\\
&=1\otimes\dots\otimes 1\otimes\sum_{r=2n-2l}^{m+(2n-2l)+1}(-1)^rc_{2n-2l+1}\otimes\dots\otimes c_rc_{r+1}\otimes\dots\otimes c_{m+(2n-2l)+1}\\
&=1\otimes\dots\otimes 1\otimes\sum_{r=0}^{m+1}(-1)^rc_{0}\otimes\dots\otimes c_rc_{r+1}\otimes\dots\otimes c_{m+1}.
\end{align*} 
From that we deduce that $\big(a_0\otimes x_0+y_0\varepsilon\big)\otimes\dots\otimes \big(a_m\otimes x_m+y_m\varepsilon\big)$ is an $m$-boundary. By K\"unneth's formula and claim \emph{i)} of the proposition, the induced map $f_*$ is an isomorphism. Then, by virtue of \cite[Corollary 2.2.3]{Lod13}, the map $f$ induces an isomorphism in cyclic homology which combined with Proposition \ref{hhaj}, \emph{ii)} yields the first isomorphism in claim \emph{ii)}. The second one is implied by the isomorphism
\[\hc_{\bullet}(\Weyl_{\hbar}(\bb C^l\oplus\bb C^{l*})[T^{-1}]\rtimes H\otimes\bb C[\varepsilon])\cong\oplus_{\gamma\in\Conj(H)}\hc_{\bullet-2k_{\gamma}}(\bb K[\varepsilon])\] 
in \cite[(A.13)]{PPT07} and the isomorphism $\hc_{p}(\bb K[\varepsilon])\cong\big(\hc_{p}(\bb K)\oplus\bb K\cdot[\varepsilon^{\otimes p}]\big)$, $p\in\bb Z_{\geq0}$, in \cite[(A.4)]{PPT07}. In the former isomorphism, we use that $\mc{D}_{\textrm{alg.}}(\bb C^l)\rtimes G\lau{\hbar}$ is isomorphic to the localization of $\Weyl_{\hbar}(\bb C^l\oplus\bb C^{l*})\rtimes G$ with respect to $T$, where $\Weyl_{\hbar}(\bb C^l\oplus\bb C^{l*})$ is the homogenization of the Weyl algebra $\Weyl(\bb C^l\oplus\bb C^{l*})$ over $\bb C\llbracket\hbar\rrbracket$ by means of the single indeterminate $\hbar$ with $\deg(\hbar)=1$.  
\end{proof}
\section{Trace densities and hypercohomology}
\label{sec4}
In this section, we generalize the standard Engeli-Felder trace density construction \cite{EF08} for the sheaf of holomorphic differential operator $\mc{D}_X$ on a complex manifold $X$ to the case of an actual and formal deformation of the sheaf of skew-group algebras $p_*\mc{D}_X\rtimes G$ on $Y$. The main ingredient, which allows us to do this vast generalization, is the construction of $\mc{H}_{1, c, X, G}$ via formal geometry in \cite{, Vit19, PhDVit19}. For completeness, we review the basics of Gelfand-Kazhdan's formal geometry in Section \ref{formalgeom} and the formal geometric construction of $\mc{H}_{1, c, X, G}$ in Section \ref{formalgeomconstrofcherednik}. This construction coupled with a  modification of the Engeli-Felder trace density map in \cite{RT12} yields a trace density morphism for the sheaf of Cherednik algebras $\mc{H}_{1, c, X, G}$ from \cite{Eti04}. This trace density map construction goes through in the case when the complex valued parameters $c$ are replaced by formal parameters $\hbar$. We use these maps to identify the hypercohomology of the Hochschild chain complex of $\mc{H}_{1, \hbar, X, G}$ with the Chen-Ruan orbifold cohomology with coefficients in the ring $\bb C\llbracket\hbar\rrbracket$.  
\subsection{Hochschild $(2n-2l)$-cocycle of the algebras $\HC$ and $\mc{A}_{n-l, l}^{H, \hbar}$}
\label{nonzerofunctionals}
Assume that the subgroup $H$ of $G$ is such that $\hh_{\bullet}(\widehat{H}_{1, c}(\bb C^l, H))$ is nontrivial. In that case, the algebra $\widehat{H}_{1, c}(\bb C^l, H)$ admits at least one nontrivial  trace. 
\begin{remark}
\label{nonzerotracegr}
For instance, when $H=S_n$, $(n\geq 2)$ and for generic values of $c$, the trace group of $\widehat{H}_{1, c}(\bb C^l, H)$ is nontrivial (see, e.g., \cite{BEG04}). 
\end{remark}
Apart from the case of a cyclic parabolic subgroup $H$ of $G$, by Corollary \ref{dimoftracegroupofformalcher}, $\hh_0(\widehat{H}_{1, \hbar}(\bb C^l, H))$ has nontrivial linear functionals when $H$ is a well-generated finite complex reflection group with no $H$-fixed points on $\bb C^l$. This guaranties the existence of sufficiently many non-zero trace density morphisms in the formal case.

Let $\phi$ and $\phi^{\hbar}$ be nontrivial linear traces of $\widehat{H}_{1, c}(\bb C^l, H)$ and $\widehat{H}_{1, \hbar}(\bb C^l, H)$, respectively.  As the identity in the rational Cherednik algebra is not a commutator,  one can safely assume that $\phi(\id)=1$ and $\phi^{\hbar}(\id)=1$.

Now,  with the help of K\"unneth's theorem, we define a Hochschild $(2n-2l)$-cocycle $\psi_{2n-2l}$ of $\mc{A}_{n-l, l}^H$ by
\begin{equation}
\label{cocycle}
\psi_{2n-2l}(a_0\otimes b_0\otimes\dots\otimes a_{2n-2l}\otimes b_{2n-2l}):=\tau_{2n-2l}(a_0\otimes\dots\otimes a_{2n-2l})\phi(b_0\dots b_{2n-2l})
\end{equation} 
for $a_0,\dots, a_{2n-2l}\in\widehat{\mc{D}}_{n-l}$, $b_0,\dots, b_{2n-2l}\in\widehat{H}_{1, c}(\bb C^l, H)$ where $\tau_{2n-2l}$ is the $\mf{gl}_{n-l}(\bb C)$-basic reduced Hochschild $(2n-2l)$-cocycle of $\widehat{\mc{D}}_{n-l}$ employed in \cite{EF08}. With the same notation, we define a Hochschild $(2n-2l)$-cocycle $\psi_{2n-2l}^{\hbar}$ of $\mc{A}_{n-l, l}^{H, \hbar}$  by
\begin{equation}
\label{formalcocycle}
\psi_{2n-2l}^{\hbar}(a_0\otimes b_0\otimes\dots\otimes a_{2n-2l}\otimes b_{2n-2l}):=\tau_{2n-2l}^{\hbar}(a_0\otimes\dots\otimes a_{2n-2l})\phi^{\hbar}(b_0\dots b_{2n-2l})
\end{equation}
where $\tau_{2n-2l}^{\hbar}$ is the obvious extension of $\tau_{2n-2l}$ to a $\bb C\llbracket\hbar\rrbracket$-linear map from $(\mc{D}_{n-l}^{\hbar})^{\hat{\otimes} 2n-2l}$ to the dual $\mc{D}_{n-l}^{\hbar*}$. Consequently, the cocycle $\tau_{2n-2l}^{\hbar}$ is $\mf{gl}_{n-l}(\bb C)$-basic and reduced. The following proposition is analogous to \cite[Proposition $4$]{RT12}. 
\begin{proposition}
\label{basiccocycles}
The $(2n-2l)$-cocycles $\psi_{2n-2l}$ and $\psi_{2n-2l}^{\hbar}$ are $\mf{gl}_{n-l}(\bb C)\oplus\mf{z}$-basic. %and $(\mf{gl}_{n-l}(\bb C)\oplus\mf{z})\otimes\bb K$-basic, respectively.
\end{proposition}
\begin{proof}
The proof is by verification and is the same as the proof of \cite[Proposition 4]{RT12}.
\end{proof}
\subsection{Review of Gelfand-Kazhdan's formal geometry}
\label{formalgeom}
In its essence, none of the material in this section is original. The presentation here repeats almost verbatim the summary of Gelfand-Kazhdan's formal geometry in \cite{Vit19, PhDVit19} where the reader can also find further relevant references on the subject.

We call the set of points $x$ in $X$ with stabilizer $\Stab(x)=H$ an isotropy type of type $H$ and denote it by $X_H$. It is a locally closed, not necessarily connected submanifold in $X$. Let $I$ be a finite set indexing the connected components of $X_H$. Then, for every $i\in I$, we denote by $X_H^i$ the $i$-th connected component of $X_H$ with  canonical inclusion map  $j_H^i:X_H^i\hookrightarrow X$. We denote by $X_i^H$ the unique connected component of the closed fixed point submanifold $X^H$ of $H$ in $X$ with canonical inclusion $j_i^H:X_i^H\hookrightarrow X$ containing the stratum $X_H^i$. As $G$ is finite, its action defines a stratification of $X$ whose strata are exactly the connected components of the isotropy types in $X$. For an exposition of stratified spaces, we refer the reader, e.g., to \cite{OR04}.

As the restricted tangent bundle $TX|_{X^H}$ of $X$ to the fixed point submanifold $X^H$ is $H$-equivariant, the normal quotient bundle $\pi: \N\to X^H$ is identified with a subbundle of $TX|_{X^H}$ such that $TX^H\oplus\N\cong TX|_{X^H}$ where $TX^H$ is the tangent bundle to $X^H$. As the group action of $G$ is per assumption faithful, the subgroup $H$ acts faithfully on the fibers of $\N$. Hence, on each connected component $X_i^H$ of codimension $l$, the subgroup can be embedded in $\GL_l(\bb C)$. Let $Z$ denote the centralizer of the image of the embedding of $H$ in $\GL_l(\bb C)$ and let $\mf{z}$ be the corresponding Lie algebra.

Now, let $\mc{N}$ be the locally free $\OO_{X_i^H}$-module corresponding to the restriction of the normal bundle $\N$ to the connected component $X_i^H$ with $\codim X_H^i=l$. Let $ \coor{\mc{N}}$ denote the set of pairs of a closed immersion of $\bb C$-ringed spaces $\Phi_x:=(\varphi, \varphi^{\#}): (0, \widehat{\OO}_{n-l})\to(X_i^H, \OO_{X_i^H})$ with $x=\varphi(0)$ and an isomorphism of $\widehat{\OO}_{n-l}$-modules $f: \widehat{\OO}_{n-l}^{\oplus l}\to\varphi^*\mc{N}$ where $\widehat{\OO}_{n-l}$ is the ring of formal functions in a formal neighborhood at the origin of $\bb C^{n-l}$. The projection $\coor{\pi}: \coor{\mc{N}}\to X_i^H$, given by $(\Phi_x, f)\mapsto x$, turns $\coor{\mc{N}}$ into a fiber bundle over $X_i^H$ with fiber at $x$ bijective to the set of infinite jets $[\phi]_x$ of parametrizations $\phi:  \bb C^{n-l}\times \bb C^{l}\to\N$ at $0$ with $\phi(0, 0)\in\N_x$. Let $\bb G:=\Aut_{n-l}\times Z(\widehat{\OO}_{n-l})$ be the pro-Lie group in which $\Aut_{n-l}:=\Aut(\widehat{\OO}_{n-l})$ and $Z(\widehat{\OO}_{n-l})$ is the group of formal power series in the coordiantes ${\bf x}=(x_1, \dots, x_{n-l})$ of $\bb C^{n-l}$ with coefficients matrices in $Z$. It acts freely and transitively on the fiber of $\coor{\mc{N}}$ from the right which makes $\coor{\mc{N}}$ a principal $\bb G$-bundle. Let $W_{n-l}:=\Der(\widehat\OO_{n-l})$ be the Lie algebra of vector fields in the formal neighboorhhod of $0$ in $\bb C^{n-l}$. Gelfand-Kazhdan's formal geometry stipulates a fiberwise isomorphism between the tangent space $T_{(\Phi_x, f)}\coor{\mc{N}}$ and the Lie algebra semidirect sum $W_{n-l}\rtimes\mf{z}\otimes\widehat{\OO}_{n-l}$. This induces a flat holomorphic $\bb G$-equivariant connection 1-form $\omega$ with values in $W_{n-l}\rtimes\mf{z}\otimes\widehat{\OO}_{n-l}$ which in turn gives $\coor{\mc{N}}$ the structure of a transitive Harish-Chandra $(W_{n-l}\rtimes\mf{z}\otimes\widehat{\OO}_{n-l}, \bb G)$-torsor over $X_i^H$. 

As $\GL_{n-l}(\bb C)\times Z$ is a closed Lie subgroup of $\bb G$, the projection $\bb G\to\bb G/\big(\GL_{n-l}(\bb C)\times Z\big)$ defines a principal $\GL_{n-l}(\bb C)\times Z$-bundle. Hence, the projection map
\begin{align}
\label{princbund}
\coor{\mc{N}}\cong\coor{\mc{N}}\times_{\bb G}\bb G\longrightarrow\coor{\mc{N}}\times_{\bb G}\bb G/\big(\GL_{n-l}(\bb C)\times Z\big)\cong\coor{\mc{N}}/\big(\GL_{n-l}(\bb C)\times Z\big)
\end{align}
is a principal $\GL_{n-l}(\bb C)\times Z$-bundle. The total space $\coor{\mc{N}}$ of \eqref{princbund} is a homogenous principal $W_{n-l}\rtimes\mf{z}\otimes\widehat{\OO}_{n-l}$-space. As the Lie algebra action commutes with the action of $\GL_{n-l}(\bb C)\times Z$ in a way compatible with the Harish-Chandra pair $(W_{n-l}\rtimes\mf{z}\otimes\widehat{\OO}_{n-l}, \GL_{n-l}(\bb C)\times Z)$, the principal bundle $\coor{\mc{N}}\to\coor{\mc{N}}/\big(\GL_{n-l}(\bb C)\times Z\big)
$ is in fact a transitive Harish-Chandra $(W_{n-l}\rtimes\mf{z}\otimes\widehat{\OO}_{n-l}, \GL_{n-l}(\bb C)\times Z)$-torsor.

When $G$ is trivial, there is only one stratum-the manifold $X$ itself. In that case, the normal bundle to $X$ is of rank $0$ and the definition of $\coor{\mc{N}}$ reduces to the standard definition of the bundle of formal coordinate systems $\coor{X}$ on $X$ (see Section $3$ in \cite{BK04}). Correspondingly, as a set, $\coor{X}$ consists of all closed immersions of $\bb C$-ringed spaces $(\varphi, \varphi^{\#}): (0, \widehat{\OO}_n)\to(X, \OO_X)$ with $x=\varphi(0)$. Similarly to $\coor{\mc{N}}$, the bundle  $\coor{X}$ has the structure of a Harish-Chandra $(W_n, \Aut_n)$-torsor over $X$ with a flat holomorphic $\Aut_n$-equivariant connection $1$-form with values in $W_n$. Furthermore, the map $\coor{X}\to\coor{X}/\GL_{n}(\bb C)$ defines a Harish-Chandra $(W_n, \GL_n(\bb C))$-torsor (\emph{cf}. \cite[Theorem $4.13$, $(4)$]{Ye05} and Section $6.1.3$ in \cite{Alm14}). 
\subsection{Review of the formal geometric construction of $\mc{H}_{1, c, X, G}$} 
\label{formalgeomconstrofcherednik}
What follows, is a succinct recollection of the construction of the sheaf of Cherednik algebras $\mc{H}_{1, c, X, G}$ by means of formal geometry in \cite{Vit19, PhDVit19}. Here, we adhere to the structure of the presentation in \cite{Vit19}.
  
Now, let $(u_i)$, $(y_i)$ be bases of $\bb C^{l}$ and its dual, respectively, let the parameter $\lambda_{A, s}$ be as defined in \cite[Lemma 4.3]{Vit19}. It is shown in \cite[Proposition $4.4$]{Vit19} that the map 
\begin{align}
 \label{liealghom}
 \Phi_c: W_{n-l}\rtimes\mf{z}\otimes\widehat{\OO}_{n-l}&\to\mc{A}_{n-l, l}^H\\    
v+A\otimes p&\mapsto v\otimes\id+p\otimes\varphi_c(A)\nonumber
\end{align}
with $\varphi_c(A)=-\sum_{i, j}A_{ij}y_ju_i+\sum_{s\in\mc{S}}\frac{2c(s)}{1-\lambda_s}\lambda_{A, s}(\id_G-s)$ is a Lie algebra embedding. Then, the Lie algebra representation $\aad\circ\Phi_c: W_{n-l}\rtimes\mf{z}\otimes\widehat{\OO}_{n-l}\to\End(\mc{A}_{n-l, l}^H)$, where $\aad$ is the adjoint action, gives $\mc{A}_{n-l, l}^H$ the structure of a Harish-Chandra $(W_{n-l}\rtimes\mf{z}\otimes\widehat{\OO}_{n-l}, \GL_{n-l}(\bb C)\times Z)$-module. The localization per \cite{BK04} of the Harish-Chandra module $\mc{A}_{n-l, l}^H$ with respect to the torsor  \eqref{princbund} is equivalent to a  holomorphic $\GL_{n-l}(\bb C)\times Z$-equivariant vector bundle $\coor{\mc{N}}\times\mc{A}_{n-l, l}^H\to\coor{\mc{N}}$ with a flat holomorphic $\GL_{n-l}(\bb C)\times Z$-equivariant connection $\nabla_H^i:=d+\aad(\Phi_c\circ\omega)(\cdot)$ with values in $\mc{A}_{n-l, l}^H$. Let $\coor{\pi}_*\OO_{\textrm{flat}}(\coor{\mc{N}}\times\mc{A}_{n-l, l}^H)|_{X_H^i}$ denote the sheaf of flat sections of that bundle, restricted to $X_H^i$. From now one, for the sake of brevity, we write $W_{x, H}^i:=W_x\cap X_H^i$. Let $\widehat{W}_x$ denote the  completion of $W_x$ with respect to the analytic subset $W_{x, H}^i:=W_x\cap X_H^i$. The first main result in the gluing procedure is the following theorem (see \cite[Theorem $5.10$]{Vit19}).
\begin{theorem}
For every parabolic subgroup $H$ of $G$ and every $H$-invariant linear slice $W_x$, there is an isomorphism of $\bb C$-algebras 
\[\mc{Y}_{H, W_x}^i: j_{H*}^i\big(\coor{\pi}_*\OO_{\textrm{flat}}(\coor{\mc{N}}\times\mc{A}_{n-l, l}^H)|_{X_H^i}\big)(W_x)\cong \OO_{\widehat{W}_x}(W_{x, H}^i)\otimes_{\OO(W_x)}H_{1, c}(W_x, H)\]
such that for every $H$-invariant linear slice $W_{x'}$ with $W_{x'}\subset W_x$, we have $\mc{Y}_{H, W_x}^i|_{W_{x'}}=\mc{Y}_{H, W_{x'}}^i$.
\end{theorem}
In the following, we describe the method by means of which the sheaves $j_{H*}^i\big(\coor{\pi}\OO_{\textrm{flat}}(\coor{\mc{N}}\times\mc{A}_{n-l, l}^H)|_{X_H^i}\big)$ on the strata of various codimensions can be  glued into a single sheaf on $X$ in the $G$-equivariant topology. The stratification defines a finite increasing filtration of $X$ into $G$-invariant open subsets $F^0(X)=\mathring{X}\subset F^1(X)\subset\cdots\subset F^{l_{\max}}(X)=X$ where $\mathring{X}$ is the principal stratum, $F^{k}(X)$ is the open disjoint union of strata of codimension less or equal to $k$. For every $k$, we successively define a sheaf $\mc{S}^k$ on $F^k(X)$ by gluing $\mc{S}^{k-1}$ with all $j_{H*}^i\big(\coor{\pi}\OO_{\textrm{flat}}(\coor{\mc{N}}\times\mc{A}_{n-l, l}^H)|_{X_H^i}\big)$ on the strata of codimension $k$ in $X$. The gluing is implemented locally on $H$-invariant linear slices. For the gluing conditions we utilize the fact that there is a morphism of Harisch-Chandra $(W_{n-l}\rtimes\mf{z}\otimes\widehat{\OO}_{n-l}, \GL_{n-l}(\bb C)\times Z)$-modules 
\[\id\otimes\widehat{\Theta}_c: \mc{A}_{n-l, l}^H\to\widehat{\mc{D}}_{n-l}\hat{\otimes} \widehat{\mc{D}}_l[\delta^{-1}]\rtimes H\] 
where $\delta=\prod_{s\in\mc{S}}\alpha_s\in\bb C[\bb C^l]$ is the discriminant and $\widehat{\Theta}_c$ is the completed Dunkl embedding. We begin by setting $\mc{S}^0:=\mc{D}_{\mathring{X}}\rtimes G$. Next, for every  basic open set $\ind_H^GW_x$ with $x\in X_H^i$ and $\codim(X_H^i)=1$, we define the set $\mc{S}^1(\ind_H^GW_x)$ of pairs of sections 
\begin{align*}
&\sum_{(g, g')\in G/H\times G/H}g\otimes p_{gg'}\otimes g'\in\bb CG\otimes_{\bb CH}(\coor{\pi}_*\OO_{\textrm{flat}}(\coor{\mathring{X}}\times\widehat{D}_n)(W_x\setminus X_H^i)\rtimes H)\otimes_{\bb CH}\bb CG,\\
&\sum_{(g, g')\in G/H\times G/H}g\otimes s_{gg'}\otimes g'\in\bb CG\otimes_{\bb CH}\coor{\pi}_*\OO_{\textrm{flat}}(\coor{\mc{N}}\times\mc{A}_{n-l, l}^H)|_{X_H^i}(W_{x, H}^i)\otimes_{\bb CH}\bb CG
\end{align*} 
saturating the conditions
\begin{align}
\label{glcondcodimzeroone}
&1.\quad \mc{Y}(p_{gg'})\in\OO(W_x)[\mc{R}(W_x)^{-1}]\otimes_{\OO(W_x)}\mc{D}_X(W_x)\rtimes H,\quad \textrm{for all}~ (g, g')\in G/H\times G/H,\nonumber\\
&2.\quad \id\otimes\widehat{\Theta}_c(s_{gg'}([\phi]_{\psi(0, 0)}))=i_{\psi}( \mc{Y}(p_{gg'}))\quad \textrm{for all}~ (g, g')\in G/H\times G/H, \infty-\textrm{jets}~~[\phi]\in\coor{\mc{N}}. 
 \end{align}
Here, $\mc{R}(W_x)$ is the multiplicative subset of $\OO(W_x)$ comprised of $1$ and all holomorphic  functions $f: W_x\to\bb C$ with $f|_{W_x\cap D}=0$ and $f(p)\neq0$ for all $p\in W_x\setminus D$ where $D$ is defined as in Section \ref{global}. Moreover, $\psi$ is a parametrization of $W_x$ with $\psi(0)=\pi(\phi(0, 0))\in W_x$ and $i_{\psi}$ denotes the Taylor expansion with respect to $({\bf x}, {\bf y})=(0, *)$ (see Section \ref{global}). By \cite[Proposition 6.1]{Vit19}, the collection of algebras $\mc{S}^{1}(\ind_H^GW_x)$ induces a sheaf $\mc{S}^1$ in the $G$-equivariant topology of $F^1(X)$. By \cite[Proposition 6.2]{Vit19}, there exists  an isomorphism of sheaves of algebras
\begin{align*}
\mf{X}^1: \mc{S}^1\to\mc{H}_{1, c, F^1(X), G}.
\end{align*}
in the $G$-equivariant topology of $X$. By induction, for every $k$ with $2\leq k\leq l_{\max}\leq n$ and every basic open set $\ind_K^GW_x$ with $x\in X_K^j$ and $\codim(X_K^j)=k$, we define $\mc{S}^k(\ind_K^GW_x)$ analogously as the set of pairs of sections
\begin{align*}
q&\in\mc{S}^{k-1}(\ind_K^G(W_x\setminus X_K^j)),\\
\sum_{(g, g')\in G/K\times G/K}g\otimes s_{gg'}\otimes g'&\in\bb CG\otimes_{\bb CK}\coor{\pi}_*\OO_{\textrm{flat}}(\coor{\mc{N}}\times\mc{A}_{n-l, l}^K)|_{X_K^j}(W_{x, K}^i)\otimes_{\bb CK}\bb CG
\end{align*} 
satisfying the gluing condition
\begin{align}
\label{glcontwo}
\sum_{(g, g')\in G/K\times G/K}g\otimes(\id\otimes\widehat{\Theta}_c)\left(s_{gg'}\big([\phi]_{\psi(0, 0)}\big)\right)\otimes g'=\sum_{(g, g')\in G/K\times G/K}g\otimes i_{\psi}(d_{gg'})\otimes g'
\end{align}
where $\mf{X}^{k-1}(q)=\sum_{(g, g')\in G/K\times G/K}g\otimes d_{gg'}\otimes g'$ with $d_{gg'}\in H_{1, c}(W_x\setminus X_K^j, K)$ according to \cite[Corollary B.8]{Vit19}. This definition induces a presheaf $\mc{S}^k$ in the $G$-equivariant topology on $X$. The main result is the following theorem (see \cite[Theorem 6.3]{Vit19}). 
\begin{theorem}
\label{maingluingresult}
For every integer $k$ with $1\leq k \leq l_{\max}$, the assignment $\ind_H^GW_x\mapsto\mc{S}^{k}(\ind_H^GW_x)$ defines a sheaf of algebras $\mc{S}^k$ such that $\mc{S}^k\cong\mc{H}_{1, c, F^k(X), G}$ as sheaves of algebra in the $G$-equivariant topology of $F^k(X)$. 
\end{theorem}
An immediate consequence of Theorem \ref{maingluingresult} is that for every stratum $X_H^i$ of codimension $l$, $0\leq l\leq n$, there is a map of sheaves
\begin{align}
\label{collapsingmap}
\rho: \mc{H}_{1, c, X, G}\to j_{H*}^i\big(\coor{\pi}_*\OO_{\textrm{flat}}(\coor{\mc{N}}\times\mc{A}_{n-l, l}^H)|_{X_H^i}\big)
\end{align}
in the $G$-equivariant topology of $X$ which we call \emph{collapsing map}. Concretely, for every $\ind_H^GW_x$ in $\mc{B}_{X}^G$ with $x\in X_H^i$, the collapsing map \eqref{collapsingmap} is given by
\begin{align}
\label{projonslices}
\rho: \Gamma (\ind_H^GW_x, \mc{H}_{1, c, X, G})&\longrightarrow\coor{\pi}_*\OO_{\textrm{flat}}(\coor{\mc{N}}\times\mc{A}_{n-l, l}^H)|_{X_H^i}(W_{x, H}^i)\\\nonumber  
(q, \sum_{(g, g')\in G/H\times G/H} g\otimes s_{gg'}\otimes g')&\mapsto\id_G\otimes s_{HH}\otimes \id_G
\end{align}
where $t_{HH}$ is the section representing the left coset of the identity in $G$. This is a well-defined map. For a $\ind_K^GW_x$ in $\mc{B}_{X}^G$ with $x$ lying on a stratum $X_K^j$,  which is contained in the closure of  $X_H^i$, the assignment  
\begin{equation}
\label{projonarbitraryopens}
\Gamma(\ind_K^GW_x, \mc{H}_{1, c, X, G})\rightarrow\fl(W_{x, H}^i)
\end{equation}
is defined in a more subtle fashion.With the help of \cite[Corollary B.9]{Vit19}, the map \eqref{projonarbitraryopens} can be expressed as the composition of the ensuing maps
\begin{align} 
\label{compofres}
&H_{1, c}(\ind_K^GW_x, G)\xrightarrow{\cong} \bb CG\otimes_{\bb CK}H_{1, c}(W_x, K)\otimes_{\bb CK}\bb CG\rightarrow H_{1, c}(W_x\setminus X_K^j, K)\nonumber\\
&\xrightarrow{\cong}\invlim_{\mf{B}_X^G}H_{1, c}(\ind_L^KW_y, K)\xrightarrow{\cong}\Big\{(s_y)\in\prod_{\substack{\ind_L^KW_y\in\mf{B}_X^G\\\ind_L^KW_y\subseteq W_x}} H_{1, c}(\ind_L^KW_y, K):~\res_{\ind_{L_1}^KW_{y'}}^{\ind_{L_{2}}^KW_y}(s_y)=s_{y'}, W_{y'}\subseteq W_y\Big\}\nonumber\\
&\rightarrow\Big\{(s_y)\in\prod_{\mf{B}_X^G\ni\ind_H^KW_y\subseteq W_x} H_{1, c}(\ind_H^KW_y, K):~\res_{\ind_{H}^KW_{y'}}^{\ind_{H}^KW_y}(s_y)=s_{y'}, W_{y'}\subseteq W_y\Big\}\nonumber\\
&\xrightarrow{\cong}\Big\{(s_y)\in\prod_{\mf{B}_X^G\ni\ind_H^KW_y\subseteq W_x}\bb CK\otimes_{\bb CH}H_{1, c}(W_y, H)\otimes_{\bb CH}\bb CK:~\res_{W_{y'}}^{W_y}(s_y)=s_{y'}, W_{y'}\subseteq W_y\Big\} \nonumber\\
&\twoheadrightarrow\Big\{(s_y)\in\prod_{y\in W_{x, H}^i} H_{1, c}(W_y, H):~\res_{W_{y'}}^{W_y}(s_y)=s_{y'}, W_{y'}\subseteq W_y\Big\}\rightarrow\fl(W_{x, H}^i)
 \end{align}
where in the last line, $W_y$ are $H$-invariant linear slices in $W_x\setminus X_K^j$ and the two-headed arrows stand for surjective maps. The last map in the composition \eqref{compofres} is defined as follows. By gluing condition \eqref{glcontwo}, each element $s_y$ in $H_{1, c}(W_y, H)$ is uniquely represented by a pair of sections
\[(q|_{W_y\setminus X_H^i}, \hat{s}|_{W_{y, H}^i})\in\mc{S}^{l}(W_y\setminus X_H^i)\times\fl(W_{y, H}^i).\] 
For all $y', y''\in W_{x, H}^i$ and $H$-invariant linear slices $W_{y'}, W_{y''}\subset W_x$, the sections $\hat{s}'|_{W_{y', H}^i}$ and $\hat{s}''|_{W_{y'', H}^i}$ coincide on every open set $W_{y''', H}^i$, contained in the intersection $W_{y', H}^i\cap W_{y'', H}^i$. By the the axioms of sheaves, there is a unique  section $\hat{s}|_{W_{x, H}^i}$ of $\fl$ over $W_{x, H}^i$ which restricts to each $\hat{s}|_{W_{y, H}^i}$. Hence, the assignment \eqref{projonarbitraryopens} is well-defined. The well-definition of maps \eqref{projonslices} and \eqref{projonarbitraryopens} shows that the collapsing morphism \eqref{collapsingmap} is a well-defined map of sheaves on the basis $\mf{B}_{X}^G$ and hence after taking projective limit a well-defined map of sheaves in the $G$-equivariant topology of $X$. We remark that in the special case, when $H=\{\id_G\}$, the corresponding collapsing map is completely determined by the definition of \eqref{projonslices} and \eqref{compofres}. 
\subsection{Construction of the trace density maps}
\label{constructionoftdm}
The construction of the sheaf of Cherednik algebras, reviewed in Section \ref{formalgeomconstrofcherednik}, allows us  to define trace densities following the construction methods of Section $2.3$ in \cite{EF08} and of Sections $4.2$, $4.4$ and $4.5$ in \cite{RT12}. Since these methods are considered standard by now, we  outline the main steps and go into details only where new phenomena appear. In this section, we follow the notation set in Section \ref{global}. 

Given a sheaf $\mc{F}$ of locally convex algebras (in the $G$-equivariant topology of $X$), the assignment $\Sh\big(U\mapsto\widehat{\C}_{\bullet}(\mc{F}(U))\big)$, where $\widehat{\C}_{\bullet}$ denotes the completed Hochschild chain complex and $\Sh$ is the sheafification functor, defines the Hochschild chain complex of the sheaf $\mc{F}$ on $X$ (in the $G$-equivariant topology). In what follows, we abuse notation by writing $\mc{C}_{\bullet}(\mc{F})$ for the above defined complex of sheaves. 

On the principal stratum $\mathring{X}$, the collapsing map \eqref{collapsingmap} gives rise to the map of complexes of left $\bb C_X$-modules
\begin{equation}
\label{principalcollapsingfmap}
\mc{C}_{\bullet}(\mc{H}_{1, c, X, G})\to\mc{C}_{\bullet}\left(j_{\id_G*}\coor{\pi}_*\OO_{\textrm{flat}}\big(\coor{\mathring{X}}\times\widehat{\mc{D}}_n\big)\right).
\end{equation}
in the $G$-equivariant topology of $X$. We successively compose the chain morphism \eqref{principalcollapsingfmap}  with the isomorphism $\mc{C}_{\bullet}\big(\coor{\pi}_*\OO_{\textrm{flat}}(j_{\id_G*}\coor{\mathring{X}}\times\widehat{\mc{D}}_n)\big)\cong\mc{C}_{\bullet}\big(j_{\id_G*}\mc{D}_{\mathring{X}}\big)$, induced by \cite[Proposition $5.3$]{Vit19}, and with Engeli-Felder's trace density morphism $(2)$ from \cite{EF08}. The resulting chain morphism in the $G$-equivariant topology on $X$ is 
\begin{equation}
\label{principaltracedensityequivtop}
\mc{C}_{\bullet}(\mc{H}_{1, c, X, G})\to j_{\id_G*}\Omega_{\mathring{X}}^{2n-\bullet}.
\end{equation}
Unfortunately, we do not know how to explicitly extend Morphism \eqref{principaltracedensityequivtop} to the whole of $X$ in the category of complexes of left $\bb C_X$-modules in the $G$-equivariant topology of $X$. Therefore, we apply the idea of the proof in Section $4.4$ in \cite{RT12} and extend the map \eqref{principaltracedensityequivtop} in the derived category ${\bf{D}}(\bb C_{X})$ of complexes of left $\bb C_X$-modules in the $G$-equivariant topology of $X$, instead. To that aim, we consider the following composition of chain maps in ${\bf{D}}(\bb C_{X})$:
\begin{align}
\label{derivedcomposition}
j_{\id_G*}\Omega_{\mathring{X}}^{2n-\bullet}\to Rj_{\id_G*}\Omega_{\mathring{X}}^{2n-\bullet}\cong R^{2n}j_{\id_G*}\bb C_{\mathring{X}}[2n]\cong\Omega_X^{2n-\bullet}.
\end{align}
As in the proof in Section $4.4$ in \cite{RT12}, the first isomorphism in the composition \eqref{derivedcomposition} follows from the fact that $\bb C_{\mathring{X}}[2n]\cong\Omega_{\mathring{X}}^{2n-\bullet}$ in ${\bf{D}}(\bb C_{\mathring{X}})$ and that $\bb C_{\mathring{X}}$ is an injective $\bb C_{\mathring{X}}$-module wherefore $Rj_{\id_G*}\Omega_{\mathring{X}}^{2n-\bullet}\cong Rj_{\id_G*}\bb C_{\mathring{X}}[2n]\cong R^{2n}j_{\id_G*}\bb C_{\mathring{X}}[2n]$. The second isomorphism in the composition \eqref{derivedcomposition} is due to the fact that for any union $U$ of (closed) submanifolds of (real) codimension $2$ and above, the de Rham cohomology groups of $X$ and $X\setminus U$ up to degree $\codim U-2$ are isomorphic. This fact implies that $R^{2n}j_{\id_G*}\bb C_{\mathring{X}}[2n]\cong{\bf{H}}^{2n}(j_{\id_G*}\Omega_{X}^{2n-\bullet})\cong\Omega_X^{2n-\bullet}$. A composition of Morphism \eqref{principaltracedensityequivtop} with the composition \eqref{derivedcomposition} in ${\bf D}(\bb C_X)$ yields in the $G$-equivariant topology on $X$ the map
\begin{equation*}
%\label{localizedtracedensity}
\mc{C}_{\bullet}(\mc{H}_{1, c, X, G})\to\Omega_{X}^{2n-\bullet}.
\end{equation*}
It is equivalent to the cochain map 
\begin{equation}
\label{localizedtracedensity}
\mc{C}_{\bullet}(\mc{H}_{1, c, X, G})\to p_*\Omega_{X}^{2n-\bullet}
\end{equation}
in the derived category ${\bf D}(\bb C_{Y})$ of $\bb C_{Y}$-modules on $Y$. We call Morphism  \eqref{localizedtracedensity} in ${\bf D}(\bb C_{Y})$ the \emph{trace density morphism} associated to the trivial subgroup of $G$.  

Let $\bb F_X:=\bb C_X\lau{\hbar_1}\cdots\lau{\hbar_k}$ be the sheaf of locally constant $\bb C\lau{\hbar_1}\cdots\lau{\hbar_k}$-valued functions on $X$ where $\bb C\lau{\hbar_1}\cdots\lau{\hbar_k}$ is the field of formal multivariate Laurent series. We recall that this sheaf has a fine resolution $\bb F_X\to\Omega_X^{\bullet}\lau{\hbar_1}\cdots\lau{\hbar_k}$ in the category of $\bb F_X$-modules. A stepwise localization of $\mc{H}_{1, \hbar, X, G}$ with respect to $\hbar_1$, \dots, $\hbar_{k-1}$ and $\hbar_k$ yields the sheaf of algebras $\mc{H}_{1, \lau{\hbar_1}, \dots, \lau{\hbar_k}, X, G}$. By repeating varbatim the steps from the previous paragraph, we obtain the cochain map
\begin{equation}
\label{formallocalizedtracedensity}
\mc{C}_{\bullet}(\mc{H}_{1, \lau{\hbar_1},\dots, \lau{\hbar_k}, X, G})\to p_*\Omega_{X}^{2n-\bullet}\lau{\hbar_1}\cdots\lau{\hbar_k}
\end{equation}
in ${\bf D}(\bb F_Y)$ on $Y$. 
Since $\bb C\llbracket\hbar\rrbracket\subset\bb C\lau{\hbar_1}\cdots\lau{\hbar_k}$ and by definition, the cocycle $\tau_{2n}^{\hbar}$ is $\bb C\llbracket\hbar\rrbracket$-linear and the image of $\tau_{2n}^{\hbar}$ lies fully in $\bb C\llbracket\hbar\rrbracket$, we can restrict the localizations on the left and right hand side of Morphism \eqref{formallocalizedtracedensity} to formal power series in $\hbar_1, \dots, \hbar_k$ in the same fashion as in Section $4.5$ in \cite{RT12}. This way, Morphism \eqref{formallocalizedtracedensity} restricts to the map
\begin{equation}
\label{formaltracedensity}
\chi_{i, \hbar}^{\id_G}: \mc{C}_{\bullet}(\mc{H}_{1, \hbar, X, G})\to p_*\Omega_{X}^{2n-\bullet}\llbracket\hbar\rrbracket
\end{equation}
in the derived category ${\bf D}(\bb C_{Y}\llbracket\hbar\rrbracket)$ of $\bb C_{Y}\llbracket\hbar\rrbracket$-modules which we call the \emph{formal trace density map} associated to the trivial subgroup of $G$. 
%Here, we use the natural identification $\bb C\llbracket\hbar_1\rrbracket\cdots\llbracket\hbar_k\rrbracket\cong\bb C\llbracket\hbar_1, \dots, \hbar_k\rrbracket$ and the usual short notation $\hbar$ for $\hbar_1, \dots, \hbar_k$.  

Let $X_i^H$ be the connected component of the fixed point submanifold of $H$ in $X$ containing the stratum $X_H^i$ in $X$ with $\codim(X_H^i)=l$, $l\geq 1$.
% and let $j_i^H:X_H^i\hookrightarrow X$ be the  canonical inclusion. 
For the purpose of defining a trace density morphism, we need to extend the collapsing map \eqref{collapsingmap} to $X_i^H$. 
%All of the relevant ingredients of formal geometry, discussed in Subsection \ref{formalgeom}, can trivially be extended to $X_i^H$.  
\begin{lemma}
 \label{extcollapsingmap}
Collapsing map \eqref{collapsingmap} has a unique extension 
 %\begin{equation*}
 $\bar\rho: \mc{H}_{1, c, X, G}\rightarrow j_{i*}^H\coor{\pi}_*\OO_{\textrm{flat}}(\coor{\mc{N}}\times\mc{A}_{n-l, l}^H)$ 
 %\end{equation*}
to $X_i^H$ in the $G$-equivariant topology of $X$.
\end{lemma}
\begin{proof}
%We distinguish two cases: ${\bf {1}})$ the codimension of the stratum $X_H^i$ is  equal to or bigger than $1$, ${\bf {2}})$ the stratum $X_H^i$ is the prinicipal (dense and open) stratum in $X$.
Let $W_x$ be $K$-invariant linear slice centered on a stratum $X_K^j$ contained in $X_i^H$ as above. The image of Morphism \eqref{compofres} is contained in the image of the surjective map 
\begin{align*}
&H_{1, c}(W_x\setminus X_K^j, H)\xrightarrow{\cong}\invlim_{W_y, y\in X_L\cap W_x, L< H}H_{1, c}(\ind_L^HW_y, H)\twoheadrightarrow\{(s_y)\in\prod_{y\in W_{x, H}^i} H_{1, c}(W_y, H):~\res_{W_{y'}}^{W_y}(s_y)=s_{y'}\}\nonumber\\
&\rightarrow\fl(W_{x, H}^i).
\end{align*}
Hence, the preimage of every section $\hat{s}$ in the image of \eqref{compofres} is non-empty in $H_{1, c}(W_x\setminus X_K^j, H)$. Furthermore, as the codimension of $X_K^j$ is at least $2$ in $X$, by Hartog's Theorem, it follows that $H_{1, c}(W_x\setminus X_K^j, H)\cong H_{1, c}(W_x, H)$. Hence, by gluing condition \eqref{glcontwo}, each representative of the preimage of $\hat{s}$ in $H_{1, c}(W_x, H)$ determines a unique section $\hat{s}_1$ of $j_{i*}^H\coor{\pi}_*\OO_{\textrm{flat}}(\coor{\mc{N}}\times\mc{A}_{n-l, l}^H)(W_x)$ such that $\hat{s}_1|_{W_x\setminus X_K^j}=\hat{s}$. By the identity theorem, all sections $\hat{s}_1$  coincide on the open subset $W_{x, H}^i$ of $W_x\cap X_i^H$, hence, on $W_x\cap X_i^H$. This gives a well-defined extension $\bar\rho $.
% 
%${\bf {2}})$ Assume that $X_H^i=\mathring{X}$. We extend Morphism \eqref{collapsingmap} step by step to the union of $\mathring{X}$ with all strata of codimension $1$. Extensions beyond codimension $1$ follow from Hartog's theorem. Assume that $\codim(X_K^j)=1$. We only need to consider the case of basic open sets $\ind_K^GW_x$ with $x$ in $X_K^j$.  Each element $u_c$ in $H_{1, c}(W_x, K)$ is uniquely represented by a pair of sections $\hat{u}_c=(t|_{W_x\setminus X_K^j}k, s_c|_{W_{x, K}^j})$ of $\coor{\pi}_*\OO_{\textrm{flat}}(\coor{\mathring{X}}\times\widehat{\mc{D}}_n)(W_x\setminus X_K^j)\rtimes K$ and $\coor{\pi}_*\OO_{\textrm{flat}}(\coor{\mc{N}}\times \mc{A}_{n-1, 1}^K)|_{X_K^j}(W_{x, K}^j)$ satisfying the gluing conditions \eqref{glcondcodimzeroone}. Setting $c(Y, s)=0$ for all complex reflections in $K$, the pair $\hat{u}_0$ describes %a unique element $u_0$ in $\mc{D}_X(W_x)\rtimes K$.  If we set $Y:=\mathring{X}\coprod X_K^j$, the element $u_0$ uniquely corresponds to an element in $\hat{r}k$ in $\coor{\pi}_*\OO_{\textrm{flat}}(\coor{Y}\times\mc{D}_n)(W_x)\rtimes K$. Clearly, the restriction maps satisfy $\res_{W_x\setminus X_K^j}^{W_x}(\hat{r}k)=\res_{W_x\setminus X_K^j}^{W_x}(\hat{u}_0)=t|_{W_x\setminus X_K^j}k$. Hence, $\hat{r}$ is the unique flat section in $\coor{\pi}_*\OO_{\textrm{flat}}(\coor{Y}\times\mc{D}_n)(W_x)$, assigned to $u_0$, which %extends $t|_{W_x\setminus X_K^j}$. This gives the wanted extension $\bar\rho$.       
\end{proof}
Lemma \ref{extcollapsingmap} induces a morphism between Hochschild chain complexes of sheaves 
\begin{equation}
\label{morphismofhhchaincomplexes}
\mc{C}_{\bullet}(\mc{H}_{1, c, X, G})\longrightarrow\mc{C}_{\bullet}(j_{i*}^H\coor{\pi}_*\OO_{\textrm{flat}}(\coor{\mc{N}}\times\mc{A}_{n-l, l}^H)) 
\end{equation}
in the $G$-equivariant topology of $X$ for every nontrivial parabolic subgroup $H$ of $G$. Let $\nabla^{\infty}$ be the $\GL_{n-l}(\bb C)\times Z$-equivariant smooth flat connection of the underlying smooth complex bundle of the  holomorphic  bundle $\coor{\mc{N}}\times\HC$ which is compatible with the holomorphic flat connection $\nabla_H^i$ defined in Section \ref{formalgeomconstrofcherednik}. Then, ~the sheaf $\efl$ is isomorphic to the sheaf $\sfl$ of flat smooth sections with respect to $\nabla^{\infty}$. By formal geometry, the fiber of $\coor{\mc{N}}/(\GL_{n-l}(\bb C)\times Z)$ is contractible. Hence, there is a global smooth section $\varphi: X_i^H\rightarrow \coor{\mc{N}}/(\GL_{n-l}(\bb C)\times Z)$. Then, $\varphi^*\nabla^{\infty}$ is a flat smooth connection on $E:=F_{\textrm{ext}}(\N)\times_{\GL_{n-l}(\bb C)\times Z}\HC$ over $X_i^H$ where $F_{\textrm{ext}}(\N)$ denotes the smooth extended frame bundle of $\N$ as defined in Section $2.3$ in \cite{EF08}. Over trivializing sets $U$ on $X_i^H$, we also have $\varphi^*\nabla^{\infty}\big|_{U}=d+[\vartheta\big|_U,~\cdot~]$ where $\vartheta\big|_U=\varphi^*(\Phi_c\circ\omega)\big|_U\in\Omega^1(U, \HC)$. Every section of $\sfl$ determines a unique flat section of $E\to X_i^H$. Thus, locally, for every $H$-invariant linear slice $W_x$ in $X$ with $x\in X_i^H$, there is a composition of maps (see Section 3.2 in \cite{Ram11}) 
\begin{align}
\label{standardmap}
\C_p(\Gamma_{\textrm{flat}}((\coor{\pi})^{-1}(W_{x, H}^i),  &\coor{\mc{N}}\times\mc{A}_{n-l,l}^H))
\hookrightarrow\C_p\Big(\big(\Omega^{\bullet}(W_{x, H}^i, \mc{A}_{n-l, l}^H), d+[\vartheta|_{W_{x, H}^i}, \cdot]\big)\Big)\nonumber\\
&\to\C_p\Big(\big(\Omega^{\bullet}(W_{x, H}^i, \mc{A}_{n-l, l}^H), d\big)\Big)\to\big(\Omega^{2n-2l-p}(W_{x, H}^i), (-1)^{2n-2l-p}d\big)\nonumber\\&
\to(\Omega^{2n-2l-p}(W_{x, H}^i), d\big)
\end{align}
where~ we~ implicitly~ use~ the~ identification~ of~ differential-graded~ algebras~ $\big(\Omega^{\bullet}(W_{x, H}^i, E), \varphi^*\nabla^{\infty}\big)\cong\big(\Omega^{\bullet}(W_{x, H}^i, \mc{A}_{n-l, l}^H), d+[\vartheta|_{W_{x, H}^i}, \cdot]\big)$. Map \eqref{standardmap} is given by 
\begin{align*}
&(\hat{s}_0, \dots, \hat{s}_p)\mapsto
(\hat{s}_0, \dots, \hat{s}_p)\\
&\mapsto\sum_{k\geq0}(-1)^{k}(\hat{s}_0, \dots, \hat{s}_p)\times(\vartheta|_{W_{x, H}^i})^k\xmapsto{\psi_{2n-2l}(\cdot)}\sum_{k\geq0}(-1)^{k}\psi_{2n-2l}\big((\hat{s}_0, \dots, \hat{s}_p)\times(\vartheta|_{W_{x, H}^i})^k\big)\\
&\xmapsto{(-1)^{\lfloor\frac{2n-2l-p}{2}\rfloor}}
\sum_{k\geq0}(-1)^{\lfloor\frac{k}{2}\rfloor}\psi_{2n-2l}\big((\hat{s}_0, \dots, \hat{s}_p)\times(\vartheta|_{W_{x, H}^i})^k\big).
\end{align*}
Here, by abuse of notation, we denote by $\hat{s}$ a flat section of $\coor{\mc{N}}\times\mc{A}_{n-l, l}^H$ as well as the corresponding flat section of $E$. We denote by $(\hat{s}_0, \dots, \hat{s}_p)$ the normalized graded Hochschild $p$-chain and $\times$ is the shuffle product of Hochschild chains. We denote by $(\vartheta|_{W_{x, H}^i})^k$ the normalized graded Hochschild $k$-chain $(1, \vartheta|_{W_{x, H}^i}, \dots, \vartheta|_{W_{x, H}^i})$. The definition of Morphism \eqref{standardmap} is almost identical to morphism $(34)$ in \cite{RT12}. 
%As $\hat{s}_0, \dots, \hat{s}_p$ are zero-forms with values in $ \mc{A}_{n-l, l}^H$, the $(p, 2n-2l-p)$-shuffle product yields a smooth $\mc{A}_{n-l, l}^H$-valued $(2n-2l-p)$-form. 
Similarly to \cite[Proposition $2.9$]{EF08} and the discussion in Section $4.2$ in \cite{RT12}, different choices of trivialization $(U_{\alpha}, \psi_{\alpha})$ of $TX_i^H\oplus\N$ change $\vartheta$ by an element in $\Omega^1(U_{\alpha\beta}, \mf{gl}_{n-l}(\bb C)\oplus\mf{z})$, $U_{\alpha\beta}:=U_{\alpha}\cap U_{\beta}\subset X_i^H$. However, this leaves the definition of Morphism \eqref{standardmap} unchanged due to the fact that by Proposition \ref{basiccocycles}, the cocycle $\psi_{2n-2l}$ is $\mf{gl}_{n-l}(\bb C)\oplus\mf{z}$-basic. Consequently, the map \eqref{standardmap} extends to the level of sheaves and combined with Morphism  \eqref{morphismofhhchaincomplexes} gives rise to the ensuing cochain map of complexes of sheaves in the $G$-equivariant topology of $X$:
\begin{equation}
\label{tdmgequivtop}
\chi_{i}^H: \mc{C}_{\bullet}(\mc{H}_{1, c, X, G})\longrightarrow j_{i*}^H\Omega_{X_i^H}^{2n-2l-\bullet}.
\end{equation}
Here, the Hochschild chain complex is turned into a cochain complex by inverting the homological degrees. Morphism \eqref{tdmgequivtop} is equivalent to the following map of cochain complexes of sheaves on $Y$: 
\begin{equation}
\label{tdm}
\chi_{i}^H: \mc{C}_{\bullet}(\mc{H}_{1, c, X, G})\longrightarrow p_*j_{i*}^H\Omega_{X_i^H}^{2n-2l-\bullet}.
\end{equation}
We refer to Morphism \eqref{tdm} as a \emph{trace density map}. If $G=\{\id_G\}$, the cocycle $\psi_{2n-2l}$ reduces to Feigin-Felder-Shoikhet's cocycle $\tau_{2n}$ from Section $2.1$ in \cite{EF08} (\emph{cf}. \cite{FFS05}). This way, the morphism \eqref{tdm} (equivalently, the map \eqref{tdmgequivtop}) reduces to Engeli-Felder's map $\chi: \mc{C}_{\bullet}(\mc{D}_X)\to\Omega_X^{2n-\bullet}$ (see Morphism $(2)$ in \cite{EF08}). 
%\[\chi: \mc{C}_{\bullet}(\mc{D}_X)\to\Omega_X^{2n-\bullet}.\]

In much the same way as the construction of Morphism \eqref{tdmgequivtop}, we obtain a map  for the Hochschild chain complex of the sheaf of formal Cherednik algebras in the $G$-equivariant topology on $X$:
\begin{equation}
\label{formaltdmgequivtop}
\chi_{i, \hbar}^H: \mc{C}_{\bullet}(\mc{H}_{1, \hbar, X, G})\longrightarrow j_{i*}^H\Omega_{X_i^H}^{2n-2l-\bullet}\llbracket\hbar\rrbracket.
\end{equation}
Morphism \eqref{formaltdmgequivtop} is naturally equivalent to the map of complexes of sheaves
\begin{equation}
\label{formaltdm}
\chi_{i, \hbar}^H: \mc{C}_{\bullet}(\mc{H}_{1, \hbar, X, G})\longrightarrow p_*j_{i*}^H\Omega_{X_i^H}^{2n-2l-\bullet}\llbracket\hbar\rrbracket 
\end{equation}
on $Y$ which is referred to as \emph{a formal trace density map}. With the help of the formal  trace density maps \eqref{formaltracedensity} and \eqref{formaltdm}, we establish the following isomorphism in the derived category ${\bf D}(\bb C_Y\llbracket\hbar\rrbracket)$ of $\bb C_Y\llbracket\hbar\rrbracket$-modules.   
\begin{theorem}
\label{quasiiso}
The map of cochain complexes
\begin{equation}
\label{formalhypercohomology}
\bigoplus_{\substack{i\\g\in G}}\chi_{i, \hbar}^g: \mc{C}_{\bullet}(\mc{H}_{1, \hbar, X, G})\rightarrow\big(\bigoplus_{\substack{i\\g\in G}}p_*j_{i*}^g\Omega_{X_{i}^g}^{2n-2l_g^i-\bullet}\llbracket\hbar\rrbracket\big)^G,
\end{equation}
where $l_g^i=\codim (X_i^g)$, is an isomorphism in ${\bf D}(\bb C_{Y}\llbracket\hbar\rrbracket)$. 
\end{theorem}
\begin{proof}
It is enough to show that the map \eqref{formalhypercohomology} induces an isomorphism at the level of cohomology sheaves. 

For each group $g\in G$, the fiber $\bb C^l$ of the normal bundle to $X^g$ does not contain the trivial representation of $\langle g\rangle$. 
%Hence, as a linear map on $\bb C^{l_g^i}$, the generator $g$ has only eigenvalues which are roots of unity unequal to $1$. 
Hence, there are nontrivial linear functionals on $\hh_0(\widehat{H}_{1, \hbar}(\bb C^l, \langle g\rangle))$.  Hence, the maps $\chi_{i, \hbar}^g$ are non-zero. As per definition of the basis $\mc{B}_X^G$, for every $x\in X_i^H$, there is a contractible $H$-invariant slice $W_x$ in $X$, an $H$-invariant contractible set $V$ in the product topology of $\bb C^n$, $n=\dim X$, containing the origin of $\bb C^n$, and an $H$-equivariant biholomorphism $f: W_x\rightarrow V$ with $f(x)=0$. The differential of $f$ equips $\bb C^n$ with the structure of an $H$-representation. Similarly, each $\bb C_{[g]}^n:=[g, \bb C^n]\subset G\times_H\bb C^n$ becomes a $gHg^{-1}$-space. As a result, the induction set $\ind_H^G\bb C^n=G\times_H\bb C^n=\coprod_{g\in G/H}\bb C_{[g]}^n$ acquires a natural right $H$-action. Moreover, each translate $gW_x$ possesses a $gHg^{-1}$-equivariant biholomorphism  $f_{[g]}$ from $gW_x$ to an open set $V_{[g]}$ in $\bb C_{[g]}^n$,  given by $f_{[g]}(y)=[g, f(g^{-1}y)]$ for every $y\in gW_x$. Hence, there is a $G$-equivariant biholomorphism $F$ from $\ind_H^GW_x$ to $\ind_H^GV$ given by $F(y)= f_{[g]}(y)$ for every $y\in gW_x$ and every $g\in G/H$. We prove the equivalent statement of Theorem \ref{quasiiso} in the $G$-equivariant topology on $X$. To that aim, it suffices to show that the homology presheaves on both sides of \eqref{formalhypercohomology} are isomorphic in the basis $\mc{B}_X^G$ of the $G$-equivariant topology. By an identical argumentation as in Proposition \ref{hhaj}, \emph{i)}, we have for the homology presheaf on every $\ind_H^GW_x\in\mc{B}_X^G$ on the left hand side of Equation \eqref{formalhypercohomology} that 
\begin{equation}
\label{genhhaj}
\hh_{\bullet}(H_{1, \hbar}(\ind_H^GW_x, G))\cong\hh_{\bullet}(\mc{D}(\ind_H^GW_x)\rtimes G)\llbracket\hbar\rrbracket.
\end{equation} 
Then, the $G$-equivariant biholomorphism $F$ and Remark \ref{hhofholdiffops}, $ii)$ imply that 
\begin{equation}
\label{Fandpropsandwich}
\hh_{\bullet}(\mc{D}(\ind_H^GW_x)\rtimes G)\llbracket\hbar\rrbracket\cong\hh_p(\mc{D}(\ind_H^GV)\rtimes G)\llbracket\hbar\rrbracket\cong\hh_{\bullet}(\mc{D}_{\textrm{alg.}}(\ind_H^G\bb C^n)\rtimes G)\llbracket\hbar\rrbracket.
\end{equation}
On the other hand, denote by $Z_H(h)$ and $C_H(h)$ the centralizer of an element $h$ in $H$, respectively its conjugacy class in $H$ and by $\h_{\textrm{dR}}$ the algebraic de Rham cohomology. Let $l_h:=\codim(\bb C)^h$. Then, with Frobenius' reciprocity theorem we simplify 
\begin{align}
\label{isobetweenhhandderham}
\hh_{\bullet}(\mc{D}_{\textrm{alg.}}(\ind_H^G\bb C^n)\rtimes G)\llbracket\hbar\rrbracket
&\cong\big(\oplus_{\substack{i\\g\in G}}\h^{2n-2\codim(\ind_H^G\bb C^n)_i^g-\bullet}((\ind_H^G\bb C^n)_i^g, \bb C)\llbracket\hbar\rrbracket\big)^G\nonumber\\
&\cong\big(\oplus_{g\in G/H}\oplus_{k\in gHg^{-1}}\h^{2n-2\codim(\bb C_{[g]}^n)^k-\bullet}((\bb C_{[g]}^n)^k, \bb C)\llbracket\hbar\rrbracket\big)^G\nonumber\\
&\cong\big(\bb CG\otimes_{\bb CH}\oplus_{k\in H}\h_{\textrm{dR}}^{2n-2\codim(\bb C^n)^k-\bullet}((\bb C^n)^k, \bb C)\llbracket\hbar\rrbracket\big)^G\nonumber\\
&\cong\big(\oplus_{k\in H}\h_{\textrm{dR}}^{2n-2l_k-\bullet}((\bb C^n)^k, \bb C)\llbracket\hbar\rrbracket\big)^H\nonumber\\
&\cong\oplus_{C_H(h)\in\Conj(H)}\big(\oplus_{k\in C_H(h)}\h_{\textrm{dR}}^{2n-2l_k-\bullet}((\bb C^n)^k, \bb C)\llbracket\hbar\rrbracket\big)^H\nonumber\\
&\cong\oplus_{C_H(h)\in\Conj(H)}\big(\bb CH\otimes_{\bb CZ_{H}(h)}\h_{\textrm{dR}}^{2n-2l_h-\bullet}((\bb C^n)^h, \bb C)\llbracket\hbar\rrbracket\big)^H\nonumber\\
&\cong\oplus_{C_H(h)\in\Conj(H)}\big(\h_{\textrm{dR}}^{2n-2l_h-\bullet}((\bb C^n)^h, \bb C)\llbracket\hbar\rrbracket\big)^{Z_H(h)}\nonumber\\
&\cong\oplus_{C_H(h)\in\Conj(H)}\big(\hh_{\bullet}(\mc{D}_{\textrm{alg.}}(\bb C^n), \mc{D}_{\textrm{alg.}}(\bb C^n)h)\llbracket\hbar\rrbracket\big)^{Z_H(h)}
\end{align}
where $i$ in the first line denotes the connected components, the first isomorphism follows directly from \cite[Proposition 3 and Proposition 4]{DE05}. From \cite{FT10}, we know that the homology $\hh_{\bullet}(\mc{D}_{\textrm{alg.}}(\bb C^n), \mc{D}_{\textrm{alg.}}(\bb C^n)h)$ is one-dimensional, spanned by the Hochschild  $(2n-2l_h)$-cycle 
\begin{equation*}
c_{2n-2l_h}=\sum_{\sigma\in S_{2n-2l_h}}1\otimes u_{\sigma(1)}\otimes \dots\otimes u_{\sigma(2n-2l_h)}  
\end{equation*}
where $u_{2i-1}=\partial_{x_{2i-1}}$, $u_{2i}=x_{2i}$. Hence, from the isomorphisms \eqref{genhhaj}, \eqref{Fandpropsandwich} and \eqref{isobetweenhhandderham} we conclude that $\hh_{\bullet}(H_{1, \hbar}(\ind_H^GW_x, G))$ is spanned by the vector $(c_{2n-2l_h})_{C_H(h)\in\Conj(H)}$ over $\bb C\llbracket\hbar\rrbracket$. On the other hand, there is a natural isomorphism
\begin{align*}
\hh^0(H_{1, \hbar}(\bb C^{l_h}, \langle h\rangle), H_{1, \hbar}(\bb C^{l_h}, \langle h\rangle)^*)
%&\cong\hh^0(\mc{D}_{\textrm{alg.}}(\bb C^{l_h})\rtimes\langle h\rangle, \mc{D}_{\textrm{alg.}}(\bb C^{l_h})\rtimes\langle h\rangle^*)\llbracket\hbar\rrbracket\\
&\cong\big(\oplus_{k=1}^{\ord(h)}\hh^0(\mc{D}_{\textrm{alg.}}(\bb C^{l_h}), \mc{D}_{\textrm{alg.}}(\bb C^{l_h})h^{k*})\llbracket\hbar\rrbracket\big)^{\langle h\rangle}
\end{align*}
where $\ord(h)$ is the order of $h$ in $G$. Each group $\hh^0(\mc{D}_{\textrm{alg.}}(\bb C^{l_h}), \mc{D}_{\textrm{alg.}}(\bb C^{l_h})h^{k*})$ is spanned by an $h^k$-twisted trace $\tr_{h^k}(\cdot)$, defined in \cite{Fed00}. It can be uniquely extended to a $\bb C\llbracket\hbar\rrbracket$-linear trace on $\hh^0(\mc{D}_{\textrm{alg.}}(\bb C^{l_h}), \mc{D}_{\textrm{alg.}}(\bb C^{l_h})h^{k*})\llbracket\hbar\rrbracket$. Hence, for each trace $\phi^{\hbar}$ of $H_{1, \hbar}(\bb C^l, \langle h\rangle)$, we can make the identification
\begin{equation}
\label{traceidentif}
\phi^{\hbar}=\sum_{k=1}^{\ord(h)}\lambda_k\tr_{h^k}(\cdot).
\end{equation}
For $g\in H$, set $W_{x, g}^i:=W_x\cap X_i^g$. Evaluation of the right hand side of \eqref{formalhypercohomology} yields 
\begin{align*}
&\big(\bigoplus_{\substack{i\\g\in G}}j_{i*}^g\Omega_{X_i^g}^{2n-2l_g^i-\bullet}\llbracket\hbar\rrbracket(\ind_H^GW_x)\big)^G\cong\bigoplus_{C_H(h)}(\bigoplus_{g\in C_H(h)}\Omega_{X_i^g}^{2n-2l_g^i-\bullet}\llbracket\hbar\rrbracket(W_{x, g}^i))^G\\
&\cong\bigoplus_{C_H(h)}(\Omega_{X_i^h}^{2n-2l_h^i-\bullet}\llbracket\hbar\rrbracket(W_{x, h}^i)^{Z_H(h)}.
%&=\big(\bigoplus_{C_G(h)\cap H\neq\varnothing}\bigoplus_{g\in C_G(h)}j_{i*}^g\Omega_{X_i^g}^{2n-2l_g^i-\bullet}\llbracket\hbar\rrbracket(\ind_H^GW_x)\big)^G\\
%&=\bigoplus_{C_G(h)\cap H\neq\varnothing}\big(j_{i*}^h\Omega_{X_i^h}^{2n-2l_h^i-\bullet}\llbracket\hbar\rrbracket(\ind_H^GW_x)\big)^{Z_G(h)}\\
\end{align*}
The first isomorphism follows from the fact that by Cartan's Lemma, the $H$-invariant slice $W_x$ intersects at most one connected component of each $X^g$ for which $\langle g\rangle\leq H$.  %As $C_G(h)\cap H$ is a disjoint union of conjugacy classes in $H$,  
The cohomology of the right hand side of \eqref{formalhypercohomology} is isomorphic to
\[\bigoplus_{C_H(h)}\big(\h^{2n-2l_h^i-\bullet}(W_{x, h}^i, \bb C)\llbracket\hbar\rrbracket\big)^{Z_G(h)}\]
%=\hspace{-2em}\bigoplus_{C_H(h)\in\Conj(H)}\hspace{-2em}\mathsmaller{\big(\h^{2n-2l_h^i-\bullet}(W_{x, h}^i, \bb C)\big)^{Z_H(h)}}.\]
Plugging the generator of $\hh_{\bullet}(H_{1, \hbar}(\ind_H^GW_x, G))$ into Morphism \eqref{formalhypercohomology} yields
\begin{align*}
\bigoplus_{\substack{i\\g\in G}}\chi_{i, \hbar}^g((c_{2n-2l_h})_{C_H(h)\in\Conj(H)})&=\bigoplus_{C_H(h)\in\Conj(H)}\chi_{i, \hbar}^h((c_{2n-2l_h})_{C_H(h)\in\Conj(H)})\nonumber\\
&=\big((-1)^{\floor{\frac{2n-2l_h}{2}}}\psi_{2n-2l_h}^{\hbar}(c_{2n-2l_h})\big)_{C_H(h)\in\Conj(H)}\nonumber\\
&=\big(\sum_{k=1}^{\ord(h)}(-1)^{\floor{\frac{2n-2l_h}{2}}}\lambda_k\tau_{2n-2l_h}^{\hbar}(c_{2n-2l_h})\tr_{h^k}(1)\big)_{C_H(h)\in\Conj(H)}
\end{align*}
where in the second line we used Identification \eqref{traceidentif}. The closing argument is analogous to the closing argument in the proof of \cite[Proposition 2.3]{EF08}. As $
\tau_{2n-2l_h}^{\hbar}(c_{2n-2l_h})$ is nonzero by \cite{FFS05, FT10} and $\tr_{h^k}(1)$ is non-zero for at least $k=1$, the map \eqref{formalhypercohomology} is a direct sum of rank $1$ invertible  matrices between the generators of the homology of the global Cherednik algebra on basic open sets and the generators of $\bigoplus_{C_H(h)\in\Conj(H)}\big(\h^{2n-2l_h^i-\bullet}(W_{x, h}^i, \bb C)\llbracket\hbar\rrbracket\big)^{Z_H(h)}$. Hence, the map $\bigoplus_{\substack{i\\g\in G}}\chi_{i, \hbar*}^g$ is an isomorphism of $\bb C_{Y}\llbracket\hbar\rrbracket$-modules.  
\end{proof}
\begin{remark}
In \cite{Vit20a}, we show that in the case when $\hbar=0$, the isomorphism \eqref{formalhypercohomology} in the derived category ${\bf D}(\bb C_Y)$ does not stem simply from a zig-zag map but from an actual quasi-isomorphism in the category of complexes of $\bb C_Y$-modules. 
\end{remark}
Let $\bb H^{\bullet}$ denote the  hypercohomology functor and let $\h_{CR}^{\bullet}(Y, \bb C)$ be the Chen-Ruan cohomology of $Y$ with complex coefficients.  We arrive at the following important direct consequence of Theorem \ref{quasiiso}. 
\begin{corollary}
\label{hypercohchenruan}
There is an isomorphism of $\bb C\llbracket\hbar\rrbracket$-modules $\bb H^{-\bullet}(Y,  \mc{C}_{\bullet}(\mc{H}_{1, \hbar, X, G}))\rightarrow\h_{CR}^{2n-\bullet}(Y, \bb C)\llbracket\hbar\rrbracket$.
\end{corollary}
The topological Euler characteristic $\chi(Y)$ of the global quotient orbifold $Y$ is defined \cite{HH90} as the average of  Euler characteristics of connected components of fixed point submanifolds $X^g$ in $X$:
\[\chi(Y):=\frac{1}{|G|}\sum_{g\in G}\sum_i\chi(X_i^g)\] 
where $X_i^g$ is the $i$-th connected component of $X^g$ with codimension $l_g^i$ and $i$ goes through the corresponding index set of connected components of $X^g$. On the other hand, the Euler characteristic $\chi(Y, \mc{C}_{\bullet}(\mc{D}_X\rtimes G))$ of the global quotient orbifold $Y$ with values in the cochain complex $\mc{C}^{\bullet}$ is given \cite{Dim04} by
\[\chi(Y, \mc{C}^{\bullet}):=\sum_{k\geq0}(-1)^k\dim\bb H^k(Y, \mc{C}^{\bullet}).\]
Corollary \ref{hypercohchenruan} establishes a link between the Euler characteristic of $Y$ with values in the Hochschild chain complex of $\mc{D}_X\rtimes G$ and the Euler characteristic of $Y$ which is a topological invariant of the orbifold. 
\begin{corollary}
\label{eulerchr}
The Euler characteristic  
$\chi(Y, \mc{C}_{\bullet}(\mc{D}_X\rtimes G))$ is invariant under formal deformations of $\mc{D}_X\rtimes G$ and continuous deformations of $Y$. In particular, $\chi(Y, \mc{C}_{\bullet}(\mc{D}_X\rtimes G))=|G|\cdot\chi(Y)$.
\end{corollary} 
\begin{proof}
The invariance under formal deformations follow from the fact that when $\hbar=0$, the isomorphism \eqref{formalhypercohomology}  remains an isomorphism in ${\bf D}(\bb C_{Y})$ (which in that case actually stems from an honest quasi-isomorphism in the category of complexes of left $\bb C_{Y}$-modules).  
%which implies 
%\begin{align*}
%$\dim_{\bb C}\bb H^{-k}(X/G, \mc{C}_{\bullet}(\mc{D}_X\rtimes G))=\dim_{\bb K}\bb H^{-k}(X/G, \mc{C}_{\bullet}(\mc{H}_{1, \lau{\hbar}, X, G}))$. 
%\end{align*}
The topological invariance follows directly from the calculation
\begin{align*}
\chi(Y, \mc{C}_{\bullet}(\mc{D}_X\rtimes G))&:=\sum_{k\geq 0}(-1)^k\dim_{\bb C}\bb H^{-k}(Y, \mc{C}^{-\bullet}(\mc{D}_X\rtimes G))\\
&=\sum_{k\geq0}(-1)^k\dim_{\bb C}\h_{\textrm{CR}}^{2n-k}(Y, \bb C)\\
%&=\sum_{k\geq0}(-1)^k\sum_{g\in G}\sum_{i}\h_{\textrm{dR}}^{2n-2l_g^i-k}(X_i^g)\\
&=|G|\frac{1}{|G|}\sum_{g\in G}\sum_i\sum_{k\geq0}(-1)^k\dim_{\bb C}\h_{\textrm{dR}}^{2n-2l_g^i-k}(X_i^g)\\
&=|G|\big(\frac{1}{|G|}\sum_{g\in G}\sum_i\chi(X_i^g)\big)\\
&=|G|\cdot\chi(Y).
\end{align*}
\end{proof}
\section{Algebraic index theorem}
\label{sec5}

To keep the derivation of an algebraic index theorem for formal deformations of $\mc{D}_X\rtimes G$ simple, we specialize to a $1$-parameter formal deformation $\mc{H}_{1, \hbar c, X, G}$. Here, as usual, $\hbar c$ stands for a family of parameters $\hbar c_1, \dots, \hbar c_k$ where $\hbar$ is a single formal indeterminate and $c_1, \dots, c_k$ are fixed complex-valued parameters. Extending the scalars of $\mc{H}_{1, \hbar c, X, G}$ from $\bb C\llbracket\hbar\rrbracket$ to $\bb K$, we get a sheaf of $\bb K$-algebras, denoted by $\mc{H}_{1, \lau{\hbar}, X, G}$. %Under a trace of $\Gamma(X/G, \mc{H}_{1, \lau{\hbar}, X, G})$ we understand a continuous in the $\hbar$-adic topology linear map $\Gamma(X/G, \mc{H}_{1, \lau{\hbar}, X, G})\to\bb K$ which factors through the quotient of the commutator $[\Gamma(X/G, \mc{H}_{1, \lau{\hbar}, X, G}),\Gamma(X/G, \mc{H}_{1, \lau{\hbar}, X, G})]$.   

%Composed with the inclusion of complexes $\mc{H}_{1, \lau{\hbar}, X, G}\rightarrow\mc{C}_{\bullet}(\mc{H}_{1, \lau{\hbar}, X, G})$, trace density map \eqref{formaltdm} induces at the level of the $0$-th hypercohomology groups a mapping 
%\[\h^0(X/G, \mc{H}_{1, \lau{\hbar}, X, G})\to\bb H^{2n-2l_H^i}(X/G, \pi_*j_{i\*}^H\Omega_{X_i^H}^{\bullet}\lau{\hbar})\cong\h^{2n-2l}(X_i^H, \bb K)\cong\bb K.\] 
%It factors through the quotient by the commutator $[\Gamma(X/G, \mc{H}_{1, \lau{\hbar}, X, G}),\Gamma(X/G, \mc{H}_{1, \lau{\hbar}, X, G})]$ which gives a linear trace
%\[\tr_{i,\lau{\hbar}}^H: \Gamma(X/G, \mc{H}_{1, \lau{\hbar}, X, G})\big/[\Gamma(X/G, \mc{H}_{1, \lau{\hbar}, X, G}),\Gamma(X/G, \mc{H}_{1, \lau{\hbar}, X, G})]\to\bb K\]
%on the global sections of the sheaf of formal Cherednik algebras. Clearly, in the limit $\hbar\to0$, $\tr_{i, \lau{\hbar}}^H\to\chi_{i*}^H$. This means that the trace does not distinguish between $\Gamma(X/G, \mc{D}_X\rtimes G)$ and its formal deformation. 
In the following, we prove an algebraic index theorem for the formal trace densities \eqref{formaltdm}. The presented results, in particular the proofs in this section follow closely analogous results in \cite{RT12},  \cite{FFS05} and \cite{PPT07}. The methods and techniques used are standard. Throughout the section, we adhere to the structure of and the notation in Sections $2.4$ and $4.3$ in \cite{RT12}. Thanks to the index theorem, if $X$ is compact, we can define  traces $\tr_{i, c}^H:=\int_{X_i^H}\chi_{i, \lau{\hbar}}^H$ on $\Gamma(Y, \mc{H}_{1, \lau{\hbar}, X, G})$ for the various parabolic subgroups $H$ of $G$. The obtained traces are not necessarily linearly independent but they have the property of varying for different choices of values for $c_1, \dots, c_k$. This way, each of the traces serves as a tool for detection of nontrivial deformations of $\Gamma(Y, \mc{D}_X\rtimes G)$ in direction of $\big(\oplus_{g\in G, i}\h^0(X_i^g, \bb C)\big)^G$ with $\codim (X_i^g)=1$ and $X_i^g\cap X_i^H\neq\varnothing$.
%Moreover, as the values of $\tr$ at the identity of the trivial deformation $\Gamma(Y, \mc{D}_X\rtimes G\lau{\hbar})$ and $\Gamma(X/G, \mc{H}_{1, \lau{\hbar}, X, G})$ are different, we argue that the trace $\tr$ distinguishes between a trivial deformation and a formal deformation in direction of $\big(\oplus_{\codim (X^g), i}\h^0(X_i^g, \bb C)\big)^G$ of $\mc{D}_X\rtimes G$. 
 
Fix a stratum $X_H^i$ and let in the following $\codim(X_H^i)=l$. Let us fix a number $N>>n$. Let further $\mf{g}:=\mf{gl}_N(\mc{A}_{n-l, l}^{H, \lau{\hbar}})$ and let $\mf{h}:=\mf{gl}_{n-l}(\bb K)\oplus\big(\mf{z}\oplus\mf{gl}_N(\bb C)\big)\otimes\bb K$. By Section 3.1 and 3.2 in \cite{FFS05}  the Hochschild cocycle $\psi_{2n-2l}^{\hbar}$, defined by Equation \eqref{formalcocycle}, corresponds to a unique $(2n-2l)$-Lie cocycle $\Psi_{2n-2l}\in\C^{2n-2l}(\mf{g}, \mf{h}; \mf{g}^*)$. Let the mapping $\ev_1: \C^{2n-2l}(\mf{g}, \mf{h}; \mf{g}^*)\rightarrow \C^{2n-2l}(\mf{g}, \mf{h}; \bb K)$ be the evaluation at the identity. In order to formulate an index theorem, we first have to compute the cohomology class $[\ev_1\Psi_{2n-2l}]$.

Let $W_{n-l}^{\lau{\hbar}}$ be the Lie algebra $\Der(\bb K[[x_1, \dots, x_{n-l}]])$. Let $\mf{z}(\bb K)\cong\mf{z}\otimes_{\bb C}\bb K$ be the Lie algebra of the Lie group $Z(\bb K)$. W set $\mf{n}:=(\mf{z}(\bb K)\oplus\mf{gl}_N(\bb K))$. 
%================DELETED ON 29.07========
%We recall from \cite[Proposition 4.4]{Vit19} the Lie algebra embedding 
%\begin{align*}
%\varphi_{\lau{\hbar}}: \mf{z}&\hookrightarrow\Chered[\bb C^l]{H}\\
%A&\mapsto-\sum_{i, j}A_{ij}y_ju_i+\sum_{s\in\mc{S}}\frac{2\hbar c(s)}{1-\lambda_s}\lambda_{A, s}(\id_G-s).
%\end{align*}
%Here, $(u_i)_i$ and $(y_i)_i$ are bases of $\bb C^l$ and its dual space, the set $\mc{S}$ denotes the collection of complex reflections $s$ in $H$, the parameter $\lambda_s$ is the eigenvalue of $s$ corresponding to the coroot of $s$ and $\lambda_{A, s}$ is as in \cite[Lemma 4.3]{Vit19}. 
%This map induces an obviously 
%================DELETED ON 29.07========
We recall the injective Lie algebra homomorphism  \eqref{liealghom} (see \cite[Proposition $4.4$]{Vit19})
\begin{align*}
\Phi_{\lau{\hbar}}: W_{n-l}^{\lau{\hbar}}\rtimes\mf{z}(\bb K)\otimes\hat{\OO}_{n-l}&\hookrightarrow\mc{A}_{n-l, l}^{H, \lau{\hbar}}\\
v+A\otimes p&\mapsto v\otimes1+p\otimes\varphi_{\lau{\hbar}}(A). 
\end{align*}
The  combination of the obvious Lie algebra embedding $\mf{h}\hookrightarrow W_{n-l}^{\lau{\hbar}}\rtimes\mf{n}\otimes\OO_{n-l}$ with the Lie algebra embedding $W_{n-l}^{\lau{\hbar}}\rtimes\mf{n}\otimes\OO_{n-l}\hookrightarrow\mf{g}$, given by 
\[v+(A, B)\otimes p\mapsto 1\otimes\Phi_{\mathsmaller{\lau{h}}}(v+A\otimes p)+B\otimes\Phi_{\mathsmaller{\lau{\hbar}}}(p),\] 
allows us to view $\mf{h}$ as a Lie subalgebra of $\mf{g}$. A decomposition of $\mf{g}$ into a direct sum of $\mf{h}$-modules $\mf{g}\cong\mf{h}\oplus\mf{g}/\mf{h}$ yields a projection of $\mf{h}$-modules $\pr: \mf{g}\rightarrow\mf{h}$ along $\mf{g}/\mf{h}$ which can be interpreted as an $\mf{h}$-equivariant projection. The amount by which this projection  fails to be a Lie algebra homomorphism is measured by the curvature $C\in\Hom(\bigwedge^2\mf{g}, \mf{h})$ defined in \cite{FFS05} by 
\[C(v, w):=[\pr(v), \pr(w)]-\pr([v, w])\]         
for all $v, w\in\mf{g}$. We can  define the Chern-Weil homomorphism $\chi: S^{\bullet}(\mf{h}^*)^{\mf{h}}\rightarrow\h^{2\bullet}(\mf{g}, \mf{h}; \bb K)$ by
\[\chi(P)(v_1\wedge\dots\wedge v_{2k})=\frac{1}{k !}\sum_{\substack{\sigma\in S_{2k}\\\sigma(2i-1)<\sigma(2i)}}(-1)^{\sigma}P\big(C(v_{\sigma(1)}, v_{\sigma(2)}), \dots, C(v_{\sigma(2k-1)}, v_{\sigma(2k)})\big)\] 
for every $P\in S^{\bullet}(\mf{h}^*)^{\mf{h}}$. Now, we prove several supporting propositions for the Chern-Weil homomorphisms which are needed for the computation of the cohomology class $[\ev_1\Psi_{2n-2l}]$. We conclude the section with the promised algebraic index theorem for $\mc{H}_{1, \lau{\hbar}, X, G}$ and a corollary for the case of a compact manifold $X$. 
\begin{proposition}
\label{cwhom1}
The Chern-Weil homomorphism $\mc{X}: S^{q}(\mf{h}^*)^{\mf{h}}\longrightarrow\h^{2q}(\mf{g},  \mf{h}; \bb K)$ is an isomorphism for $N>>n$ and $q\leq n-l+k$ where $k=\min_{\gamma\in\Conj(H)}k_{\gamma}$ and $k_{\gamma}$ is as in Corollary \ref{hhchofhcmodule}, ii).
\end{proposition}
\begin{proof}
Assume in what follows that $\deg(\varepsilon)=1$. By \cite[Theorem $10.2.5$]{Lod13}, we have that 
\begin{equation}
\label{theoremladay}
\h_{m}(\mf{gl}_N(\mc{A}_{n-l, l}^{H, \lau{\hbar}}\otimes\bb C[\varepsilon]))\cong\big(\mc{S}^{\bullet}(\hc_{\bullet}(\mc{A}_{n-l, l}^{H, \lau{\hbar}}\otimes\bb C[\varepsilon])[1])\big)_m
\end{equation} 
for every $m\geq0$ where $\mc{S}^{\bullet}$ is the graded symmetric product defined for instance in A.1 in \cite{Lod13} and $(\cdot)_m$ denotes the $m$-th degree of a  graded module. As explained in Appendix $A$ in \cite{PPT07}, the left hand side of Isomorphism \eqref{theoremladay} can be written as
\begin{equation}
\label{lhsofsgiso}
 \h_{m}(\mf{gl}_N(\mc{A}_{n-l, l}^{H, \lau{\hbar}}\otimes\bb C[\varepsilon]))\cong\oplus_{p=0}^m\h_p(\mf{g}, S^{m-p}\mf{g}\varepsilon)
\end{equation}
On account of Corollary \ref{hhchofhcmodule}, \emph{ii)}, the right hand side of Isomorphism \eqref{lhsofsgiso} can be written as
\begin{equation}
\label{rhsofsgiso}
\bigoplus_{d\geq0}\bigoplus_{j_1+\dots+j_d=m}\Big(\oplus_{\gamma}\hc_{j_1-(2n-2l+2k_{\gamma})-1}(\bb K[\varepsilon])\Big)\otimes\dots\otimes\Big(\oplus_{\gamma}\hc_{j_d-(2n-2l+2k_{\gamma})-1}(\bb K[\varepsilon])\Big) 
\end{equation}
The isomorphism \eqref{theoremladay} is graded of degree $0$. Hence, in particular it respects the grading in $\varepsilon$. Hence, it maps cohomology classes of degree $m-p$ in $\varepsilon$ to elements of degree $m-p$ in $\varepsilon$. Hence, inserting Isomorphisms \eqref{lhsofsgiso} and \eqref{rhsofsgiso} into Isomorphism \eqref{theoremladay} and comparing degrees of $\varepsilon$, we get for every $p\leq m$,
\begin{align*}
\h_p(\mf{g}, S^{m-p}\mf{g}\varepsilon)&\cong\big(\oplus_{\gamma}\hc_{m-(2n-2l+2k_{\gamma})-1}(\bb K[\varepsilon])\big)_{(m-p)-\textrm{th degree in}~\varepsilon}\\
&\cong\Big(\hspace{-0.3em}\oplus_{\gamma}\big(\hc_{m-(2n-2l+2k_{\gamma})-1}(\bb K)\oplus\bb K[\hspace{-2.1em}\underbrace{\varepsilon\otimes\dots\otimes\varepsilon}_{\textrm{$m-(2n-2l+2k_{\gamma})$-times}}\hspace{-2.1em}]\big)\hspace{-0.1em}\Big)_{(m-p)-\textrm{th degree in $\varepsilon$}}\\
&\cong\begin{cases}
\bigoplus_{\substack{\gamma\in\Conj(H)\\k_{\gamma}=k}}\bb K[\hspace{-0.6em}\underbrace{\varepsilon\otimes\dots\otimes\varepsilon}_{\textrm{$m-2n-2l+2k$}}\hspace{-0.6em}],\quad\textrm{if}~p=2n-2l+2k\\
0,\quad\textrm{if}~p<2n-2l+2k  
\end{cases}
\end{align*}
Since Lie algebra homology and cohomology are dual, another way of stating the above is: $\h^p(\mf{g}, S^{q}\mf{g})$ is isomorphic to $\bb K^{a_{n-l+k}}$ when $p=2n-2l+2k$ and is $0$ otherwise. The remainder of the proof follows verbatim that of \cite[Proposition 5.2]{FFS05}.  
\end{proof} 
\begin{proposition}
\label{cwhom2}
The Chern-Weil homomorphism $\chi: S^q(\mf{h}^*)^{\mf{h}}\rightarrow\h^{2q}(W_{n-l}^{\lau{\hbar}}\rtimes\mf{n}\otimes\OO_{n-l}, \mf{h}; \bb K)$ is an isomorphism for $q\leq n-l$. Furthermore,~ $\h^{2q}(W_{n-l}^{\lau{\hbar}}\rtimes\mf{n}\otimes\OO_{n-l}, \mf{h}, \bb K)=\C^{2q}(W_{n-l}^{\lau{\hbar}}\rtimes\mf{n}\otimes\OO_{n-l}, \mf{h}; \bb K)$. 
\end{proposition}
\begin{proof}
The injective Lie algebra homomorphism $W_{n-l}^{\lau{\hbar}}\rtimes\mf{n}\otimes\OO_{n-l}\hookrightarrow\mf{g}$ induces a natural $\mf{h}$-equivariant injective map $\eta: \bigwedge^{2q}\big(W_{n-l}^{\lau{\hbar}}\rtimes\mf{n}\otimes\OO_{n-l}/\mf{h}\big)\rightarrow\bigwedge^{2q}\big(\mf{g}/\mf{h}\big)$ which in turn gives an $\mf{h}$-equivariant  surjective morphism $\eta^*: \h^{2q}(\mf{g}, \mf{h};\bb K)\rightarrow \h^{2q}(W_{n-l}^{\lau{\hbar}}\rtimes\mf{n}\otimes\OO_{n-l}, \mf{h};\bb K)$. It is a straightforward verification that $\chi=\eta^*\circ\mc{X}$, where $\mc{X}$ is the Chern-Weil homomorphism from Proposition \ref{cwhom1}. Thus, the map $\chi$ is surjective. On the other hand, on account of \cite[Corollary 1]{Kho07}, we have $S^q(\mf{h}^*)^{\mf{h}}\cong\h^{2q}(W_{n-l}^{\lau{\hbar}}\rtimes\mf{n}\otimes\OO_{n-l}, \mf{h}, \bb K)$ for $q\leq n-l$. Hence, $\chi$ is in fact an isomorphism because $S^q(\mf{h}^*)^{\mf{h}}$ is a finite-dimensional $\bb K$-vector space for $q\leq n-l$. We show  that $\C^{2q+1}(W_{n-l}^{\lau{\hbar}}\rtimes\mf{n}\otimes\OO_{n-l}, \mf{h}, \bb K)=0$ making use of invariant theory the same way as in \cite{fuks86} which implies the second statement of the proposition.   
 \end{proof}
Let $X_1\oplus X_2\oplus X_3\in\mf{h}$. Let $(\hat{A}_{\hbar}\Ch_{\phi^{\hbar}}\Ch)_k\in S(\mf{h}^*)^{\mf{h}}$ be the homogeneous term of degree $k$ in the Taylor expansion of $\hat{A}_{\hbar}\Ch_{\phi^{\hbar}}\Ch(X):=\hat{A}_{\hbar}(X_1)\Ch_{\phi^{\hbar}}(X_2)\Ch(X_3)$, where $\hat{A}(X_1)=\det\Big(\frac{X/2}{\sinh(X/2)}\Big)^{1/2}$ and $\hat{A}_{\hbar}(X_1)=\hat{A}(\hbar X_1)$, $\Ch_{\phi^{\hbar}}(X_2)=\phi^{\hbar}(\exp(X_2))$ and $\Ch(X_3)=\tr(\exp(X_3))$. In the ensuing proposition, we compute the cohomology class $[\ev_1\Psi_{2n-2l}]$ following almost verbatim the proof of \cite[Theorem 3]{RT12} and imitating various techniques from the proofs of \cite[Theorem $5.1$]{FFS05} and \cite[Theorem $5.3$]{PPT07}.  
\begin{proposition}
\label{rrh}
Assume that $\hbar\neq0$. Then,
$[\ev_1\Psi_{2n-2l}]=(-1)^{n-l}\mc{X}\big((\hat{A}_{\hbar}\Ch_{\phi^{\hbar}}\Ch)_{n-l}\big)$.
\end{proposition}
\begin{proof}
By Proposition \ref{cwhom1}, there is an $\mf{h}$-invariant polynomial $P$ with $\mc{X}(P)=[\ev_1\Psi_{2n-2l}]$. The $\aad(\mf{h})$-invariance of $P$ implies that it is uniquely determined by its value on the Cartan subalgebra $\mf{a}$ of $\mf{h}$ spanned by the following group of vectors: $\delta_{ii}\in\mf{gl}_{n-l}(\bb K)$, $1\leq i\leq n-l$, $E_{r_ar_a}\in\mf{z}(\bb K)$, $1\leq r_a\leq n_a$, $1\leq a\leq t$, $E_{rr}\in\mf{gl}_N(\bb K)$, $1\leq r\leq N$. Here, we denote by $(\delta_{ij})_{i, j}$ the identity matrix $\mf{gl}_{n-l}(\bb K)$, by $t$ the number of isotypic components of the semisimple $\bb CH$-module $\bb C^l$ and by $n_a$ the multiplicities of the $t$ non-equivalent simple $\bb CH$-submodules of $\bb C^l$. We can view the  generators of $\mf{a}$ as elements in $\mf{g}$ via  the following identification: 
\begin{align*}
\delta_{ii}&\mapsto\id_N\otimes x_i\frac{d}{dx_i}\otimes1,\quad\textrm{for}\quad 1\leq i\leq n-l,\\
E_{r_ar_a}&\mapsto\hat{E}_{r_ar_a}:=\id_N\otimes1\otimes\big(\sum_{i, j}-(E_{r_ar_a})_{ij}y_ju_i+\sum_{s\in\mc{S}}\frac{2\hbar c(s)}{1-\lambda_s}\lambda_{E_{r_ar_a}, s}(\id_G-s)\big),\quad\textrm{for}\quad& 1\leq r_a\leq n_a,~1\leq a\leq t,\\
E_{rr}&\mapsto E_{rr}\otimes1\otimes1,\quad\textrm{for}\quad 1\leq r\leq N. 
\end{align*}
Consider the commutative diagram 
\begin{equation*}
\begin{tikzcd}
S^{n-l}(\mf{h}^*)^{\mf{h}} \arrow[rd, "\chi"] \arrow[r, "\mc{X}"] &\h^{2n-2l}(\mf{g}, \mf{h}; \bb K)\arrow[d, "\eta^*"] \\
&C^{2n-2l}(W_{n-l}^{\lau{\hbar}}\rtimes\mf{n}\otimes\OO_{n-l}, \mf{h}; \bb K)
\end{tikzcd}
\end{equation*} 
emanating from Proposition \ref{cwhom2}. Since all the arrows are isomorphisms here, we can prove the restriction of the desired identity to $W_{n-l}^{\lau{\hbar}}\rtimes\mf{n}\otimes\OO_{n-l}$, that is $\chi(P)=\eta^*([\ev_1\Psi_{2n-2l}])$. Due to this restriction, the identity becomes an identity of cocycles rather than of cohomology classes. To shorten the notation, throughout the proof we shall write $\ev_1\Psi_{2n-2l}$ to denote its cohomology class as well as its restriction $\eta^*([\ev_1\Psi_{2n-2l}])$ to the Lie subalgebra $\mf{n}$. Just like in Equation $(22)$ in the proof of \cite[Theorem 3]{RT12}, we select an invariant polynomial $P_{n-l}$ whose restriction to $\mf{a}\subset\mf{g}$ is given by 
\begin{align*}
P_{n-l}(M_1\otimes a_1\otimes b_1,& \dots, M_{n-l}\otimes a_{n-l}\otimes b_{n-l})=\tr(M_1\dots M_{n-l})\phi^{\hbar}(b_1\dots b_{n-l})\\
&\mu_{n-l}\int_{[0,1]^{n-l}}\prod_{1\leq i\leq j\leq n-l}\exp(\hbar\psi(u_i-u_j)\alpha_{ij})(a_1\otimes\dots\otimes a_n)du_1\dots du_{n-l}
\end{align*}
for all $M_i\otimes a_i\otimes b_i\in\mf{a}$, $i=1, \dots, n-l$.  Here, the maps $\mu_{n-l}, \psi, \alpha_{ij}$ and the variables $u_1,\dots, u_{n-l}$ are defined as in \cite[Section 2.3]{FFS05}. We evaluate the cocycle $\ev_1\Psi_{2n-2l}$ on the following special vectors 
\begin{align*}
u_{ij}&:=-\frac{1}{2}x_i^2\frac{d}{dx_i}\delta_{ij}+x_ix_j\frac{d}{dx_j},\quad\quad
v_{ir}:=x_i\otimes E_{rr},\quad\quad
w_{ir_a}:= x_i\otimes \hat{E}_{r_ar_a}
\end{align*}
where the indices are as above. These vectors are in the kernel of $\pr$ and satisfy the following commutator relations 
\begin{align*}
[\frac{d}{dx_i}, u_{ij}]&=x_j\frac{d}{dx_j},\qquad[\frac{d}{dx_i}, v_{ir}]=E_{rr},\qquad[\frac{d}{dx_i}, x_i\otimes\hat{E}_{r_ar_a}]=\hat{E}_{r_ar_a}.
\end{align*}
Hence, $C(\frac{d}{dx_i}, u_{ij})=-x_j\frac{d}{dx_j}$, $C(\frac{d}{dx_i}, v_{ir})=-E_{rr}$ and $C(\frac{d}{dx_i}, w_{ir_a})=-\hat{E}_{r_ar_a}$. In what follows, we denote by $f_i$ any vector of the form $u_{ij}$ with $i\geq j$, or $v_{ir}$ or $w_{ir_a}$. Then, 
\begin{align}
\label{cwidentity}
\chi(P)(\frac{d}{dx_1}\wedge f_1\wedge\dots\wedge\frac{d}{dx_{n-l}}\wedge f_{n-l})
=(-1)^{n-l}P(\frac{df_1}{dx_1}, \dots, \frac{df_{n-l}}{dx_{n-l}})  
\end{align} 
where in $\chi(P)$ only  permutations $\sigma\in S_{2n-2l}$ with $\sigma(i)-\sigma(i-1)=1$ for all $i=2, 4, \dots, 2n-2l$ contribute nontrivially. The number of such permutations in $S_{2n-2l}$ equals the number of permutations of pairs of tuples $(2i-1, 2i)$, $i=1, \dots, n-l$ which is exactly $(n-l)!$. Each basis vector of $\mf{a}$ is of the form $\frac{df_i}{dx_i}$ for some $f_i$. With that in mind, we show that the left hand side of \eqref{cwidentity} is $\ev_1\Psi_{2n-2l}$ exactly the same way as in the proof of \cite[Theorem 5.1]{FFS05} and that of \cite[Theorem 3]{RT12}. Namely, one gets 
\[\chi(P)(\frac{d}{dx_1}\wedge f_1\wedge\dots\wedge\frac{d}{dx_{n-l}}\wedge f_{n-l})=P_{n-l}(\frac{df_1}{dx_1}, \dots, \frac{df_{n-l}}{dx_{n-l}})\]
which combined with \eqref{cwidentity} implies $P=(-1)^{n-l}P_{n-l}$  on $\mf{a}$. Thus, $P=(-1)^{n-l}P_{n-l}$ on $\mf{h}$. It remains to calculate $P_{n-l}$ on $\mf{a}$. We start by remarking that $P_{n}=P_n'\phi^{\hbar}$ where $P_n'$ is the polynomial defined in Equation $(8)$ in the proof of \cite[Theorem 5.1]{FFS05}. In the same fashion as in \cite[Theorem 5.3]{PPT07}, we explicitly calculate $P_{n-l}$ on the diagonal matrices $X=Y+Z$ where $Y:=\sum_{i=1}^{n-l}\nu_ix_i\frac{d}{dx_1}+\sum_{r=1}^N\sigma_rE_{rr}\in\mf{gl}_{n-l}(\bb K)\oplus\mf{gl}_{N}(\bb K)$ and $Z:=\sum_{1\leq a\leq t, 1\leq r_a\leq n_a}\tau_{r_a}\hat{E}_{r_ar_a}\in\mf{z}(\bb K)$, $\nu_i, \sigma_r, \tau_{r_a}\in\bb K$. To that aim, we  consider the generating function $S(X)=\sum_{m\geq1}\frac{1}{m!}P_m(X, \dots, X)$. We then have
\begin{align*}
S(X)&=\sum_{l\geq0}\frac{1}{l!}P_l'(\underbrace{Y, \dots, Y}_{\textrm{$l$ times}})\sum_{k=m-l\geq0}\frac{1}{k!}\phi^{\hbar}(Z^{k})=(\hat{A}_{\hbar}\Ch)(Y)\Ch_{\phi^{\hbar}}(Z)
\end{align*} 
where we use the identity $\sum_l\frac{1}{l!}P_l'(Y, \dots, Y)=(\hat{A}_{\hbar}\Ch)(Y)$ in the proof of \cite[Theorem $5.1$]{FFS05} (see also \cite[Theorem $5.3$]{PPT07}). Since $P_{n-l}$ is the degree $n-l$ component of $S$, it is equal to $\big(\hat{A}_{\hbar}\Ch\Ch_{\phi^{\hbar}}\big)_{n-l}$. Hence, we have $\eta^*([\ev_1\Psi_{2n-2l}])=(-1)^{n-l}\chi\Big(\big(\hat{A}_{\hbar}\Ch\Ch_{\phi^{\hbar}}\big)_{n-l}\Big)$. The assertion follows.  
\end{proof}
Let henceforth $\varphi^*\nabla^{\infty}:=
 \nabla+[A, ~\cdot~]$ 
be the flat smooth connection on the associated vector bundle $E$ with a fiber $\mc{A}_{n-l, l}^{H, \lau{\hbar}}$ over $X_i^H$ from Section \ref{constructionoftdm} where $\nabla$ is a smooth (non-flat) connection and $A\in\Omega^1(X_i^H, E)$. Consequently, by defiintion
\[(\varphi^*\nabla^{\infty})^2=\nabla^2 + [\nabla A+\frac{1}{2}[A, A], ~\cdot~]=[\Theta,~ \cdot~]\]
with a central element $\Theta\in\Omega^2(X_i^H, \bb K)$. At the same time the curvature of the non-flat connection $\nabla$ can be written in the form  $\nabla^2=[R_T+R_N, ~\cdot~]$ with $R_T\in\Omega^2(X_i^H, \mf{gl}_{n-l}(\bb K))$ and $R_N\in\Omega^2(X_i^H, \mf{z}(\bb K))$ (see, e.g., Section $4$ in \cite{FFS05} and Section $4.1$ in \cite{RT12}) from which we conclude
\begin{equation}
\label{almostmc}
\nabla A+\frac{1}{2}[A, A]=\Theta-R_T-R_N. 
\end{equation} 
We observe that on every trivializing chart $U$ of $TX_i^H\oplus\N$ on $X_i^H$, the $1$-forms $A|_U$ and $\vartheta|_U$ from Section \ref{constructionoftdm} differ by a $\mf{gl}_{n-l}(\bb C)\oplus\mf{z}$-valued $1$-form on $U$. Hence, we can use $A$ in the definition of $\chi_{i, \lau{\hbar}}^H$. The following theorem and its proof mimic  \cite[Theorem 6]{RT12} and its proof, respectively. 
\begin{theorem}
\label{algindexthm}
For~~$\id\in\Gamma(Y, \mc{H}_{1, \lau{\hbar}, X, G})$, the $(2n-2l)$-form on $X_i^H$ \[\chi_{i, \lau{\hbar}}^H(\id)-\hbar^{n-l}\Big(\hat{A}(R_T)\Ch(\frac{-\Theta}{\hbar})\Ch_{\phi^{\hbar}}(\frac{R_N}{\hbar})\Big)_{n-l}\] is exact.
\end{theorem}
\begin{proof}
We assume that $A$ saturates  $\pr(A)=0$. We are allowed to do this because $\varphi^*\nabla^{\infty}$ can be rewritten in the form $\varphi^*\nabla^{\infty}=\nabla+[A, -]=(\nabla+[\pr(A), -])+[(A-\pr(A)), -]=\tilde{\nabla}+[\tilde{A}, -]$. Then, accounting that $\pr(\nabla A)=\nabla\pr(A)$, for any pair of smooth vector fields $\xi_1, \xi_2$ on $X_H^i$, we have 
\begin{equation}
\label{curvaturecomp}
C(A\xi_1, A\xi_2)=-\pr(\nabla A(\xi_1, \xi_2)+[A(\xi_1), A(\xi_2)])=-\pr(\Theta-R_T-R_N)=R_T+R_N-\Theta
\end{equation} 
where in the second equality, we use Equation \eqref{almostmc}. 

Denote the homogeneous $\mf{h}$-invariant polynomial $\big(\hat{A}_{\hbar}\Ch\Ch_{\phi^{\hbar}}\big)_{n-l}$ by $P$ and let $v_1,\dots, v_{2n-2l}$ be vector fields on $X$. Note that $\ev_1\Psi_{2n-2l}(A^{\otimes2n-2l})=(2n-2l)!\psi_{2n-2l}^{\hbar}((A)^{2n-2l})$ where the notation $(A)^{2n-2l}=(1, A, \dots, A)$ is as in Section \ref{constructionoftdm}. Then, in exactly the same fashion as in the proof of \cite[Theorem 6]{RT12}, we have
\begin{align*}
&\chi_{i, \lau{\hbar}}^{H}(\id)(v_1, \dots, v_{2n-2l})
=(-1)^{n-l}\ev_1\Psi_{2n-2l}(\underbrace{A\wedge\dots\wedge A}_{\textrm{$2n-2l$~times}})(v_1, \dots, v_{2n-2l})\\
&=\mc{X}(P)(\underbrace{A\wedge\dots\wedge A}_{\textrm{$2n-2l$~times}})(v_1, \dots, v_{2n-2l})\\
&=\frac{1}{(n-l)!}\sum_{\sigma}(-1)^{\sigma}P(C(Av_{\sigma(1)}, Av_{\sigma(2)}), \dots, C(Av_{\sigma(2n-2l-1)}, Av_{\sigma(2n-2l)}))\\
&=\frac{1}{(n-l)!}\sum_{\sigma}(-1)^{\sigma}P((R_T+R_N-\Theta)(v_{\sigma(1)}, v_{\sigma(2)}), \dots, (R_T+ R_N-\Theta)(v_{\sigma(2n-2l-1)}, v_{\sigma(2n-2l)}))\\
&=\frac{1}{(n-l)!}P(\underbrace{(R_T+R_N-\Theta), \dots, (R_T+R_N-\Theta)}_{\textrm{$n-l$ times}})(v_1, \dots, v_{2n-2l})\\
&=P(R_T+R_N-\Theta)(v_1, \dots, v_{2n-2l})
\end{align*}
where after the first line all equalities are modulo exact forms. In the fourth line, we apply  Equality \eqref{curvaturecomp}. The last line is implied by the same argument as in the proof of \cite[Theorem 6]{RT12}. We implicitly use in the definition of the trace density that $\psi_{2n-2l}^{\hbar}$ is $\mf{gl}_{n-l}(\bb C)\oplus\mf{z}$-basic. 
Hence, modulo exact forms, we have \[\chi_{i, \mathsmaller{\lau{\hbar}}}^H(\id)=\big(\hat{A}_{\hbar}(R_T)\Ch(-\Theta)\Ch_{\phi^{\hbar}}(R_N)\big)_{n-l}=\hbar^{n-l}\big(\hat{A}_{\hbar}(\frac{R_T}{\hbar})\Ch(\frac{-\Theta}{\hbar})\Ch_{\phi^{\hbar}}(\frac{R_N}{\hbar})\big)_{n-l}.\]
The definition of $\hat{A}_{\hbar}$ and $\hat{A}$ imply the claim. 
\end{proof}
Suppose now that $X$ is compact. Then, we can define the linear functional $\tr_{i, c}^H: \Gamma(Y, \mc{H}_{1, \lau{\hbar}, X, G})\to\bb K$ by
\[\xi\mapsto\int_{X_i^H}\chi_{i, \lau{\hbar}}^H(\xi).\]

%The map $\ev_{\hbar=0}\circ L$ is not well-defined which implies that $L$ and $\tr_{i, \lau{\hbar}}^H$ are linearly independent.
This linear functional is a trace because $\phi^{\hbar}$ is a trace of $\widehat{H}_{1, \lau{\hbar}}(\bb C^l, H)$ and $\tau_{2n-2l}^{\hbar}$ vanishes on reduced Hochschild $2n$-chains of the form $[D_0, D_1]\times(1, \omega, \dots, \omega)$ ($2n$ times $\omega$ ) where $D_0, D_1\in\widehat{\mc{D}}_n^{\lau{\hbar}}$ and $\omega$ is a Maurer-Cartan form with values in $\widehat{\mc{D}}_n^{\lau{\hbar}}$. In particular, the trace $\tr_{i, c}^H$ is well-defined in the extreme case $c_1=\dots=c_k=0$ which yields a trace for the trivial deformation $\mc{D}(X)\rtimes G\lau{\hbar}$.  
\begin{corollary}
Suppose that $c\neq0$. Then, $\tr_{i, 0}^H(\id)\neq\tr_{i, c}^H(\id)$.
\end{corollary}
\begin{proof}
This follows from the fact that $R_N$ in $\mf{z}(\bb K)$ is identified with the element $R_N^1=-\sum_{i, j}(R_N)_{ij}y_ju_i+\sum_{s\in\mc{S}}\frac{2\hbar c(s)}{1-\lambda_s}\lambda_{R_N, s}(\id_G-s)$ in $\Chered[\bb C^l]{H}$ and with the element $R_N^2=-\sum_{i, j}(R_N)_{ij}y_ju_i$ in $\widehat{\mc{D}}_l\rtimes H\lau{\hbar}$, respectively. Since the center of $\bb KH$ is not a commutator of $\Chered[\bb C^l]{H}$, we have that $\phi^{\hbar}(\exp(R_N^1))\neq\phi^{\hbar}(\exp(R_N^2))$. Hence, the statement follows. 
%Hence, $\chi_{i, \lau{\hbar}}^H(\id_1)$ and $\chi_{i, \lau{\hbar}}^H(\id_0)$ are the constant sections $1\otimes\varphi_{\hbar}(1)\in\mc{A}_{n-l, l}^{H, \lau{\hbar}}$  with 
%\[\varphi_{\hbar}(1)=-\sum_{i}y_iu_i+\sum_{s\in\mc{S}}\frac{2\hbar c(s)}{1-\lambda_s}\lambda_{\id, s}(\id_G-s)\big)\] 
%and $1\otimes(-1)\in\widehat{\mc{D}}_{n-l}^{\lau{\hbar}}\widehat{\otimes}\widehat{\mc{D}}_l\rtimes H\lau{\hbar}$, respectively. As the center of $\bb KH$ is not a commutator of $\Chered[\bb C^l]{H}$, the trace value $\phi^{\hbar}(\varphi_{\hbar}(\id_G))\neq0$. On the other hand, as in the second case the set of complex reflections $\mc{S}$ is empty, $\phi^{\hbar}(\id)=0$. The claim follows. 
\end{proof} 

The above results mean that the trace $\tr_{i, c}^H(\id)$, more generally the index $\chi_{i, \lau{\hbar}}^H(\id)$, serve as a homological detector of nontrivial deformations in similar fashion as the trace $\chi^{\gamma}$ in Section 4.2 in \cite{RT12}. %The existence of such a trace for the algebra of global sections of formal deformations implies that the representation theory of $\mc{D}(X)\rtimes G\lau{\hbar}$ and $\Gamma(X/G, \mc{H}_{1, \lau{\hbar}, X, G})$ are different.  

\section{Acknowledgements}
This note presents results derived from my Ph.D thesis \cite{PhDVit19} completed at ETH Zurich. I am indebted to my advisor Prof. Giovanni Felder and my second advisor Prof. Ajay Ramadoss for explaining to me in detail their works on trace densities, algebraic index theorems and other related research matters, as well as for teaching me the techniques which were applied in this note. I am grateful to Prof. Pavel Etingof for various helpful comments and suggestions regarding a previous version of the current manuscript as well as for explaining to me many of the subtle properties of the sheaves of Cherednik algebras. I also  thank Prof. Gwyn Bellamy and Prof. Ulrich Thiel for a fruitful email exchange on the properties of complex reflection groups. Finally, I am grateful to the anonymous referee for his careful reading of the manuscript and the suggestion of several improvements.  
\printbibliography
\vspace{1ex}
\textsc{Department of Mathematics, Massachusetts Institute of Technology, 77 Massachusetts Ave.,\\ Cambridge, MA 02139, USA}\\
\textsc{E-mail address}: \textit{avitanov@protonmail.com}
\end{document}